\documentclass{birkjour}
 \newtheorem{thm}{Theorem}[section]
 \newtheorem{cor}[thm]{Corollary}
 \newtheorem{lem}[thm]{Lemma}
 \newtheorem{prop}[thm]{Proposition}
 \theoremstyle{definition}
 \newtheorem{defn}[thm]{Definition}
 \theoremstyle{remark}
 \newtheorem{rem}[thm]{Remark}
 \newtheorem*{ex}{Example}
 \numberwithin{equation}{section}
\usepackage{cite}
\begin{document}
%-------------------------------------------------------------------------
% editorial commands: to be inserted by the editorial office
%
%\firstpage{1} \volume{228} \Copyrightyear{2004} \DOI{003-0001}
%\seriesextraline{This is the Concrete Title of this Book\br H.E. R and S.T.C. W, Eds.}
%\firstpage{1}
\title[Characterizations of BLO and Campanato spaces]
{Some new characterizations of BLO and Campanato spaces in the Schr\"{o}dinger setting}
%----------Author 1
\author{Cong Chen}
\address{School of Mathematics and Systems Science,\\
Xinjiang University,\\
Urumqi 830046, P. R. China}

\email{3221043817@qq.com.}
%----------Author 2
\author{Hua Wang}
\address{School of Mathematics and Systems Science,\\
Xinjiang University,\\
Urumqi 830046, P. R. China}

\email{wanghua@pku.edu.cn.}

\thanks{This work was supported by the Natural Science Foundation of China\\
(No. XJEDU2020Y002 and 2022D01C407)}
\date{\today}
\dedicatory{Dedicated to the memory of Li Xue.}
%%% ----------------------------------------------------------------------
\begin{abstract}
Let us consider the Schr\"{o}dinger operator $\mathcal{L}=-\Delta+V$ on $\mathbb R^d$ with $d\geq3$, where $\Delta$ is the Laplacian operator on $\mathbb R^d$ and the nonnegative potential $V$ belongs to certain reverse H\"{o}lder class $RH_s$ with $s\geq d/2$. In this paper, the authors first introduce two kinds of function spaces related to the Schr\"{o}dinger operator $\mathcal{L}$. A real-valued function $f\in L^1_{\mathrm{loc}}(\mathbb R^d)$ belongs to the (BLO) space $\mathrm{BLO}_{\rho,\theta}(\mathbb R^d)$ with $0\leq\theta<\infty$ if
\begin{equation*}
\|f\|_{\mathrm{BLO}_{\rho,\theta}}
:=\sup_{\mathcal{Q}}\bigg(1+\frac{r}{\rho(x_0)}\bigg)^{-\theta}\bigg(\frac{1}{|Q(x_0,r)|}
\int_{Q(x_0,r)}\Big[f(x)-\underset{y\in\mathcal{Q}}{\mathrm{ess\,inf}}\,f(y)\Big]\,dx\bigg),
\end{equation*}
where the supremum is taken over all cubes $\mathcal{Q}=Q(x_0,r)$ in $\mathbb R^d$, $\rho(\cdot)$ is the critical radius function in the Schr\"{o}dinger context. For $0<\beta<1$, a real-valued function $f\in L^1_{\mathrm{loc}}(\mathbb R^d)$ belongs to the (Campanato) space $\mathcal{C}^{\beta,\ast}_{\rho,\theta}(\mathbb R^d)$ with $0\leq\theta<\infty$ if
\begin{equation*}
\|f\|_{\mathcal{C}^{\beta,\ast}_{\rho,\theta}}
:=\sup_{\mathcal{B}}\bigg(1+\frac{r}{\rho(x_0)}\bigg)^{-\theta}
\bigg(\frac{1}{|B(x_0,r)|^{1+\beta/d}}\int_{B(x_0,r)}\Big[f(x)-\underset{y\in\mathcal{B}}{\mathrm{ess\,inf}}\,f(y)\Big]\,dx\bigg),
\end{equation*}
where the supremum is taken over all balls $\mathcal{B}=B(x_0,r)$ in $\mathbb R^d$. Then we establish the corresponding John--Nirenberg inequality suitable for the space $\mathrm{BLO}_{\rho,\theta}(\mathbb R^d)$ with $0\leq\theta<\infty$ and $d\geq3$. Moreover, we give some new characterizations of the BLO and Campanato spaces related to $\mathcal{L}$ on weighted Lebesgue spaces, which is the extension of some earlier results.
\end{abstract}
%%% ----------------------------------------------------------------------
\maketitle
%%% ----------------------------------------------------------------------
\tableofcontents

\section{Introduction}\label{intro}
The theory of function spaces has been a central topic in modern analysis, and new function spaces are now of increasing use in the fields such as harmonic analysis and partial differential equations. The main purpose of this paper is to give several characterizations for the new BLO and Campanato spaces in the Schr\"{o}dinger setting. Another purpose of this paper is to establish a version of John--Nirenberg inequality for the new BLO space. Let $d\geq3$ be a positive integer and $\mathbb R^d$ be the $d$-dimensional Euclidean space, and let $V:\mathbb R^d\rightarrow\mathbb R$, $d\geq3$, be a nonnegative locally integrable function which belongs to the \emph{reverse H\"older class} $RH_s(\mathbb R^d)$ with $s\in(1,\infty]$. We recall that $V\in RH_s(\mathbb R^d)$ means that there exists a positive constant $C=C(s,V)>0$ such that the following \emph{reverse H\"older inequality}
\begin{equation*}
\bigg(\frac{1}{|B|}\int_B V(y)^s\,dy\bigg)^{1/s}\leq C\cdot\bigg(\frac{1}{|B|}\int_B V(y)\,dy\bigg)
\end{equation*}
holds for every ball $B$ in $\mathbb R^d$, with the usual modification made when $s=\infty$. In particular, if $V$ is a nonnegative polynomial, then $V\in RH_{\infty}(\mathbb R^d)$. Let us consider the \emph{Schr\"{o}dinger differential operator} with the nonnegative potential $V$.
\begin{equation*}
\mathcal{L}:=-\Delta+V \quad \mbox{on}~~~ \mathbb R^d,
\end{equation*}
where $\Delta=\sum_{j=1}^d\frac{\partial^2}{\partial x_j^2}$ is the standard Laplace operator on $\mathbb R^d$. As in \cite{shen}, for any given $V\in RH_s(\mathbb R^d)$ with $s\geq d/2$ and $d\geq3$, we introduce the \emph{critical radius function} $\rho(x)=\rho(x;V)$ (determined by $V$), which is defined by
\begin{equation}\label{rho}
\rho(x):=\sup\bigg\{r>0:\frac{1}{r^{d-2}}\int_{B(x,r)}V(y)\,dy\leq1\bigg\},\quad x\in\mathbb R^d,
\end{equation}
where $B(x,r)$ denotes the open ball with the center at $x$ and radius $r$. It is well known that this auxiliary function satisfies $0<\rho(x)<\infty$ for any $x\in\mathbb R^d$ under the above assumption on $V$ (see \cite{shen}).

Throughout this paper, we will always assume that $V\not\equiv0$ and $V\in RH_s(\mathbb R^d)$ with $s\geq d/2$.
\begin{ex}
The Schr\"{o}dinger operator $\mathcal{L}=-\Delta+V$ can be viewed as a perturbation of the Laplace operator.
\begin{enumerate}
  \item When $V\equiv1$, we obtain $\rho(x)=1$ for any $x\in\mathbb R^d$.
  \item When $V(x)=|x|^2$ and $\mathcal{L}$ becomes the Hermite operator, we obtain $\rho(x)\approx(1+|x|)^{-1}$.
\end{enumerate}
\end{ex}
The notation $\mathbf{X}\approx \mathbf{Y}$ means that there exists a positive constant $C>0$ such that $1/C\leq \mathbf{X}/\mathbf{Y}\leq C$.

We need the following known result concerning the critical radius function \eqref{rho}, which was proved by Shen in \cite{shen}.
\begin{lem}[\cite{shen}]\label{N0}
If $V\in RH_s(\mathbb R^d)$ with $s\geq d/2$ and $d\geq3$, then there exist two positive constants $C_0\geq 1$ and $N_0>0$ such that
\begin{equation}\label{com}
\frac{\,1\,}{C_0}\bigg(1+\frac{|x-y|}{\rho(x)}\bigg)^{-N_0}\leq\frac{\rho(y)}{\rho(x)}
\leq C_0\bigg(1+\frac{|x-y|}{\rho(x)}\bigg)^{\frac{N_0}{N_0+1}}
\end{equation}
for all $x,y\in\mathbb R^d.$
\end{lem}
To state our main results, we first recall the definition of the classical BMO space and BLO space.

A locally integrable function $f$ on $\mathbb R^d$ is said to belong to $\mathrm{BMO}(\mathbb R^d)$, the space of bounded mean oscillation, if
\begin{equation*}
\|f\|_{\mathrm{BMO}}:=\sup_{\mathcal{Q}}\frac{1}{|\mathcal{Q}|}\int_{\mathcal{Q}}\big|f(x)-f_{\mathcal{Q}}\big|\,dx<\infty,
\end{equation*}
where the supremum is taken over all cubes $\mathcal{Q}$ in $\mathbb R^d$ and $f_{\mathcal{Q}}$ stands for the mean value of $f$ over $\mathcal{Q}$; that is,
\begin{equation*}
f_{\mathcal{Q}}:=\frac{1}{|\mathcal{Q}|}\int_{\mathcal{Q}}f(y)\,dy.
\end{equation*}
The space of BMO functions was first introduced by John and Nirenberg in \cite{john}.

A locally integrable function $f$ on $\mathbb R^d$ is said to belong to $\mathrm{BLO}(\mathbb R^d)$, the space of bounded lower oscillation, if
there exists a constant $c_1>0$ such that for any cube $\mathcal{Q}\subset\mathbb R^d$,
\begin{equation*}
\frac{1}{|\mathcal{Q}|}\int_{\mathcal{Q}}\Big[f(x)-\underset{y\in\mathcal{Q}}{\mathrm{ess\,inf}}\,f(y)\Big]\,dx\leq c_1.
\end{equation*}
The minimal constant $c_1$ as above is defined to be the BLO-constant of $f$ and denoted by $\|f\|_{\mathrm{BLO}}$. The space of BLO functions was first introduced by Coifman and Rochberg in \cite{coifman}. It is easy to see that
\begin{equation*}
L^{\infty}(\mathbb R^d)\subset\mathrm{BLO}(\mathbb R^d)\subset\mathrm{BMO}(\mathbb R^d).
\end{equation*}
Moreover, the above inclusion relations are both strict, see \cite{meng1,meng2,ou} for some examples. It is easy to verify that
\begin{equation}\label{bmoblo}
\|f\|_{\mathrm{BMO}}\leq 2\|f\|_{\mathrm{BLO}}.
\end{equation}
In fact, for any cube $\mathcal{Q}\subset\mathbb R^d$,
\begin{equation*}
\begin{split}
\frac{1}{|\mathcal{Q}|}\int_{\mathcal{Q}}\big|f(x)-f_{\mathcal{Q}}\big|\,dx
&=\frac{1}{|\mathcal{Q}|}\int_{\mathcal{Q}}\Big|f(x)
-\underset{y\in\mathcal{Q}}{\mathrm{ess\,inf}}\,f(y)+\underset{y\in\mathcal{Q}}{\mathrm{ess\,inf}}\,f(y)-f_{\mathcal{Q}}\Big|\,dx\\
&\leq\frac{1}{|\mathcal{Q}|}\int_{\mathcal{Q}}\Big[f(x)-\underset{y\in\mathcal{Q}}{\mathrm{ess\,inf}}\,f(y)\Big]\,dx
+\Big|\underset{y\in\mathcal{Q}}{\mathrm{ess\,inf}}\,f(y)-f_{\mathcal{Q}}\Big|\\
&\leq\frac{2}{|\mathcal{Q}|}\int_{\mathcal{Q}}\Big[f(x)-\underset{y\in\mathcal{Q}}{\mathrm{ess\,inf}}\,f(y)\Big]\,dx
\leq 2\|f\|_{\mathrm{BLO}},
\end{split}
\end{equation*}
as desired.
\begin{rem}
It should be pointed out that $\|\cdot\|_{\mathrm{BLO}}$ is not a norm and $\mathrm{BLO}(\mathbb R^d)$ is not a linear space (it is a proper subspace of $\mathrm{BMO}(\mathbb R^d)$).
\end{rem}
On the other hand, the classical Campanato space was studied extensively in the literature, and played an important role in the study of harmonic analysis and partial differential equations. Let $0<\beta\leq1$. A locally integrable function $f$ is said to belong to the Campanato space $\mathcal{C}^{\beta}(\mathbb R^d)$
if
\begin{equation*}
\|f\|_{\mathcal{C}^{\beta}}:=\sup_{\mathcal{B}}\frac{1}{|\mathcal{B}|^{1+\beta/d}}\int_{\mathcal{B}}\big|f(x)-f_{\mathcal{B}}\big|\,dx<\infty,
\end{equation*}
where the supremum is taken over all balls $\mathcal{B}$ in $\mathbb R^d$ and $f_{\mathcal{B}}$ stands for the mean value of $f$ over $\mathcal{B}$.
The Campanato space $\mathcal{C}^{\beta}(\mathbb R^d)$ was first introduced by Campanato in \cite{cam}. In 2007, motivated by the definition of the space $\mathrm{BLO}(\mathbb R^d)$, Hu--Meng--Yang introduced the following space $\mathcal{C}^{\beta,\ast}(\mathbb R^d)$, which is a subspace of $\mathcal{C}^{\beta}(\mathbb R^d)$. Let $0<\beta\leq1$. A locally integrable function $f$ is said to belong to $\mathcal{C}^{\beta,\ast}(\mathbb R^d)$ if there exists a positive constant $c_2>0$ such that for any ball $\mathcal{B}\subset\mathbb R^d$,
\begin{equation*}
\frac{1}{|\mathcal{B}|^{1+\beta/d}}\int_{\mathcal{B}}\Big[f(x)-\underset{y\in\mathcal{B}}{\mathrm{ess\,inf}}\,f(y)\Big]\,dx\leq c_2.
\end{equation*}
The minimal constant $c_2$ as above is defined to be the $\mathcal{C}^{\beta,\ast}$-constant of $f$ and denoted by $\|f\|_{\mathcal{C}^{\beta,\ast}}$.
\begin{rem}
\begin{enumerate}
  \item As in \eqref{bmoblo}, we also have that
  \begin{equation*}
  \mathcal{C}^{\beta,\ast}(\mathbb R^d)\subset\mathcal{C}^{\beta}(\mathbb R^d)\quad \&\quad \|f\|_{\mathcal{C}^{\beta}}\leq 2\|f\|_{\mathcal{C}^{\beta,\ast}}.
  \end{equation*}
  \item We point out that $\|\cdot\|_{\mathcal{C}^{\beta,\ast}}$ is not a norm and $\mathcal{C}^{\beta,\ast}(\mathbb R^d)$ is not a linear space (it is a proper subspace of $\mathcal{C}^{\beta}(\mathbb R^d)$).
\end{enumerate}
\end{rem}
In 2011, Bongioanni--Harboure--Salinas \cite{bong3} introduced a new class of function spaces (see also \cite{bong2}). According to \cite{bong3}, the new BMO space $\mathrm{BMO}_{\rho,\infty}(\mathbb R^d)$ is defined by
\begin{equation*}
\mathrm{BMO}_{\rho,\infty}(\mathbb R^d):=\bigcup_{\theta>0}\mathrm{BMO}_{\rho,\theta}(\mathbb R^d),
\end{equation*}
where for any fixed $0<\theta<\infty$ the space $\mathrm{BMO}_{\rho,\theta}(\mathbb R^d)$ is defined to be the set of all locally integrable functions $f$ satisfying
\begin{equation}\label{BM}
\frac{1}{|Q(x_0,r)|}\int_{Q(x_0,r)}\big|f(x)-f_{Q}\big|\,dx\leq C_1\cdot\bigg(1+\frac{r}{\rho(x_0)}\bigg)^{\theta},
\end{equation}
for all $x_0\in\mathbb R^d$ and $r\in(0,\infty)$, $f_{Q}$ denotes the mean value of $f$ on $Q(x_0,r)$.
A norm for $f\in \mathrm{BMO}_{\rho,\theta}(\mathbb R^d)$, denoted by $\|f\|_{\mathrm{BMO}_{\rho,\theta}}$, is given by the infimum of the constants satisfying \eqref{BM}, after identifying functions that differ by a constant, or equivalently,
\begin{equation*}
\|f\|_{\mathrm{BMO}_{\rho,\theta}}
:=\sup_{\mathcal{Q}}\bigg(1+\frac{r}{\rho(x_0)}\bigg)^{-\theta}\bigg(\frac{1}{|Q(x_0,r)|}\int_{Q(x_0,r)}\big|f(x)-f_{Q}\big|\,dx\bigg),
\end{equation*}
where the supremum is taken over all cubes $\mathcal{Q}=Q(x_0,r)$ with $x_0\in\mathbb R^d$ and $r\in(0,\infty)$. Note that if we let $\theta=0$ in \eqref{BM}, we obtain the classical (John--Nirenberg) BMO space. Define
\begin{equation*}
\mathrm{BMO}_{\rho,\theta}(\mathbb R^d):=\Big\{f\in L^1_{\mathrm{loc}}(\mathbb R^d):\|f\|_{\mathrm{BMO}_{\rho,\theta}}<\infty\Big\}.
\end{equation*}
With the above definition in mind, one has
\begin{equation*}
\mathrm{BMO}(\mathbb R^d)\subset \mathrm{BMO}_{\rho,\theta_1}(\mathbb R^d)\subset \mathrm{BMO}_{\rho,\theta_2}(\mathbb R^d)
\end{equation*}
whenever $0<\theta_1<\theta_2<\infty$, and hence
\begin{equation*}
\mathrm{BMO}(\mathbb R^d)\subset\mathrm{BMO}_{\rho,\infty}(\mathbb R^d).
\end{equation*}
Moreover, it can be shown that the classical BMO space is properly contained in $\mathrm{BMO}_{\rho,\infty}(\mathbb R^d)$ (see \cite{bong2,bong3,tang} for more examples).

A classical result due to John and Nirenberg in the 1960's(also known as the John--Nirenberg inequality) states that there exist two positive constants $C_1$ and $C_2$ such that for every cube $\mathcal{Q}=Q(x_0,r)$ in $\mathbb R^d$ and every $\lambda>0$, we have (see \cite{john})
\begin{equation*}
\Big|\Big\{x\in \mathcal{Q}:|f(x)-f_{\mathcal{Q}}|>\lambda\Big\}\Big|
\leq C_1|\mathcal{Q}|\exp\bigg\{-\frac{C_2\lambda}{\|f\|_{\mathrm{BMO}}}\bigg\},
\end{equation*}
when $f\in\mathrm{BMO}(\mathbb R^d)$.

In 2015, Tang proved a version of John--Nirenberg inequality suitable for the new BMO space $\mathrm{BMO}_{\rho,\theta}(\mathbb R^d)$ with $\theta>0$. His proof can be found in \cite[Proposition 4.2]{tang}.
\begin{lem}[\cite{tang}]\label{expbmo}
If $f\in \mathrm{BMO}_{\rho,\theta}(\mathbb R^d)$ with $0<\theta<\infty$, then there exist two positive constants $C_1$ and $C_2$ such that for every cube $\mathcal{Q}=Q(x_0,r)$ in $\mathbb R^d$ and every $\lambda>0$, we have
\begin{equation*}
\Big|\Big\{x\in \mathcal{Q}:|f(x)-f_{\mathcal{Q}}|>\lambda\Big\}\Big|
\leq C_1|\mathcal{Q}|\exp\bigg\{-\bigg(1+\frac{r}{\rho(x_0)}\bigg)^{-(N_0+1)\theta}\frac{C_2\lambda}{\|f\|_{\mathrm{BMO}_{\rho,\theta}}}\bigg\},
\end{equation*}
where $N_0$ is the constant appearing in Lemma \ref{N0}.
\end{lem}
In 2014, Liu--Sheng introduced a new class of function spaces which is larger than the classical Campanato space. According to \cite{liu}, for $0<\theta<\infty$ and $0\leq\beta\leq1$, the space $\mathcal{\mathcal{C}}^{\beta}_{\rho,\theta}(\mathbb R^d)$ is defined to be the set of all locally integrable functions $f$ satisfying
\begin{equation}\label{Lipliu}
\frac{1}{|B(x_0,r)|^{1+\beta/d}}\int_{B(x_0,r)}\big|f(x)-f_{B}\big|\,dx
\leq C_2\cdot\bigg(1+\frac{r}{\rho(x_0)}\bigg)^{\theta},
\end{equation}
for all $x_0\in\mathbb R^d$ and $r\in(0,\infty)$, $f_{B}$ denotes the mean value of $f$ on $B(x_0,r)$. The infimum of the constants satisfying \eqref{Lipliu} is defined to be the norm of $f\in \mathcal{\mathcal{C}}^{\beta}_{\rho,\theta}(\mathbb R^d)$ and denoted by $\|f\|_{\mathcal{C}^{\beta}_{\rho,\theta}}$, or equivalently,
\begin{equation*}
\|f\|_{\mathcal{\mathcal{C}}^{\beta}_{\rho,\theta}}:=\sup_{\mathcal{B}}\bigg(1+\frac{r}{\rho(x_0)}\bigg)^{-\theta}
\bigg(\frac{1}{|B(x_0,r)|^{1+\beta/d}}\int_{B(x_0,r)}\big|f(x)-f_{B}\big|\,dx\bigg),
\end{equation*}
where the supremum is taken over all balls $\mathcal{B}=B(x_0,r)$ with $x_0\in\mathbb R^d$ and $r\in(0,\infty)$.
\begin{rem}
\begin{enumerate}
Some special cases:
  \item Note that if $\theta=0$ in \eqref{Lipliu}, then $\mathcal{C}^{\beta}_{\rho,\theta}(\mathbb R^d)$ is exactly the classical Campanato space $\mathcal{C}^{\beta}(\mathbb R^d)$(see, for instance, \cite{cam,janson});
  \item Note that if $\beta=0$ and $0<\theta<\infty$ in \eqref{Lipliu}, then $\mathcal{C}^{\beta}_{\rho,\theta}(\mathbb R^d)$ is exactly the above space $\mathrm{BMO}_{\rho,\theta}(\mathbb R^d)$ introduced by Bongioanni--Harboure--Salinas in \cite{bong3}.
\end{enumerate}
\end{rem}
Define
\begin{equation*}
\mathcal{C}^{\beta}_{\rho,\theta}(\mathbb R^d):=\Big\{f\in L^1_{\mathrm{loc}}(\mathbb R^d):\|f\|_{\mathcal{C}^{\beta}_{\rho,\theta}}<\infty\Big\}.
\end{equation*}
Since $[1+r/{\rho(x_0)}]^{\theta}\geq1$, it is obvious that
\begin{equation*}
\mathcal{C}^{\beta}(\mathbb R^d)\subset \mathcal{C}^{\beta}_{\rho,\theta_1}(\mathbb R^d)\subset \mathcal{C}^{\beta}_{\rho,\theta_2}(\mathbb R^d)
\end{equation*}
whenever $0<\theta_1<\theta_2<\infty$. Then we write
\begin{equation*}
\mathcal{C}^{\beta}_{\rho,\infty}(\mathbb R^d):=\bigcup_{\theta>0}\mathcal{C}^{\beta}_{\rho,\theta}(\mathbb R^d).
\quad \&\quad \mathcal{C}^{\beta}(\mathbb R^d)\subset\mathcal{C}^{\beta}_{\rho,\infty}(\mathbb R^d).
\end{equation*}

Motivated by the definition of $\mathrm{BLO}(\mathbb R^d)$ and $\mathcal{C}^{\beta,\ast}(\mathbb R^d)$, we now introduce the following spaces $\mathrm{BLO}_{\rho,\theta}(\mathbb R^d)$ and $\mathcal{C}^{\beta,\ast}_{\rho,\theta}(\mathbb R^d)$ related to Schr\"{o}dinger operators.
\begin{defn}
Let $\theta\in[0,\infty)$. A locally integrable function $f$ is said to belong to the space $\mathrm{BLO}_{\rho,\theta}(\mathbb R^d)$, if there exists a positive constant $\overline{C}_1>0$ such that for any cube $\mathcal{Q}=Q(x_0,r)\subset\mathbb R^d$,
\begin{equation}\label{wangdef1}
\frac{1}{|Q(x_0,r)|}\int_{Q(x_0,r)}\Big[f(x)-\underset{y\in\mathcal{Q}}{\mathrm{ess\,inf}}\,f(y)\Big]\,dx\leq \overline{C}_1\cdot\bigg(1+\frac{r}{\rho(x_0)}\bigg)^{\theta}.
\end{equation}
The infimum of the constants $\overline{C}_1$ satisfying \eqref{wangdef1} is defined to be the $\mathrm{BLO}_{\rho,\theta}$-constant of $f$ and denoted by $\|f\|_{\mathrm{BLO}_{\rho,\theta}}$, that is,
\begin{equation*}
\|f\|_{\mathrm{BLO}_{\rho,\theta}}
:=\sup_{\mathcal{Q}}\bigg(1+\frac{r}{\rho(x_0)}\bigg)^{-\theta}
\bigg(\frac{1}{|Q(x_0,r)|}\int_{Q(x_0,r)}\Big[f(x)-\underset{y\in\mathcal{Q}}{\mathrm{ess\,inf}}\,f(y)\Big]\,dx\bigg),
\end{equation*}
where the supremum is taken over all cubes $\mathcal{Q}=Q(x_0,r)$ with $x_0\in\mathbb R^d$ and $r\in(0,\infty)$.
\end{defn}

\begin{defn}
Let $\beta\in(0,1]$ and $\theta\in[0,\infty)$. A locally integrable function $f$ is said to belong to the space $\mathcal{C}^{\beta,\ast}_{\rho,\theta}(\mathbb R^d)$, if there exists some constant $\overline{C}_2>0$ such that for any ball $\mathcal{B}=B(x_0,r)\subset\mathbb R^d$,
\begin{equation}\label{wangdef2}
\frac{1}{|B(x_0,r)|^{1+\beta/d}}\int_{B(x_0,r)}\Big[f(x)-\underset{y\in\mathcal{B}}{\mathrm{ess\,inf}}\,f(y)\Big]\,dx
\leq \overline{C}_2\cdot\bigg(1+\frac{r}{\rho(x_0)}\bigg)^{\theta}.
\end{equation}
The infimum of the constants $\overline{C}_2$ satisfying \eqref{wangdef2} is defined to be the $\mathcal{C}^{\beta,\ast}_{\rho,\theta}$-constant of $f$ and denoted by $\|f\|_{\mathcal{C}^{\beta,\ast}_{\rho,\theta}}$, that is,
\begin{equation*}
\|f\|_{\mathcal{C}^{\beta,\ast}_{\rho,\theta}}
:=\sup_{\mathcal{B}}\bigg(1+\frac{r}{\rho(x_0)}\bigg)^{-\theta}
\bigg(\frac{1}{|B(x_0,r)|^{1+\beta/d}}
\int_{B(x_0,r)}\Big[f(x)-\underset{y\in\mathcal{B}}{\mathrm{ess\,inf}}\,f(y)\Big]\,dx\bigg),
\end{equation*}
where the supremum is taken over all balls $\mathcal{B}=B(x_0,r)$ with $x_0\in\mathbb R^d$ and $r\in(0,\infty)$.
\end{defn}

Adapting the arguments in \cite[p.123]{duoand} and \cite[p.124]{grafakos2}, we can also prove the following variant of John--Nirenberg inequality: if $f\in\mathrm{BLO}(\mathbb R^d)$, then there exist two positive constants $C_3$ and $C_4$ such that for every cube $\mathcal{Q}$ and every $\lambda>0$, we have (see \cite{wangding})
\begin{equation*}
\Big|\Big\{x\in \mathcal{Q}:\Big[f(x)-\underset{y\in\mathcal{Q}}{\mathrm{ess\,inf}}\,f(y)\Big]>\lambda\Big\}\Big|\leq
C_3|\mathcal{Q}|\exp\bigg\{-\frac{C_4\lambda}{\|f\|_{\mathrm{BLO}}}\bigg\}.
\end{equation*}
Inspired by these results, we will establish a version of John--Nirenberg inequality with precise constants for the space $\mathrm{BLO}_{\rho,\theta}(\mathbb R^d)$ related to Schr\"{o}dinger operators. That is, if $f\in \mathrm{BLO}_{\rho,\theta}(\mathbb R^d)$ with $0<\theta<\infty$, then there exist two positive constants $\overline{C}_3$ and $\overline{C}_4$ such that for every cube $\mathcal{Q}=Q(x_0,r)$ and every $\lambda>0$, we have
\begin{equation*}
\begin{split}
&\Big|\Big\{x\in \mathcal{Q}:\Big[f(x)-\underset{y\in\mathcal{Q}}{\mathrm{ess\,inf}}\,f(y)\Big]>\lambda\Big\}\Big|\\
&\leq \overline{C}_3|\mathcal{Q}|\exp\bigg\{-\bigg(1+\frac{r}{\rho(x_0)}\bigg)^{-(N_0+1)\theta}\frac{\overline{C}_4\lambda}{\|f\|_{\mathrm{BLO}_{\rho,\theta}}}\bigg\},
\end{split}
\end{equation*}
where $N_0$ is the constant appearing in Lemma \ref{N0}. Based on this result, the goal of this paper is to give some new characterizations for the space $\mathrm{BLO}_{\rho,\theta}(\mathbb R^d)$ with $0<\theta<\infty$. Moreover, we also establish similar results for the space $\mathcal{C}^{\beta,\ast}_{\rho,\theta}(\mathbb R^d)$ with $0<\beta\leq 1$ and $0<\theta<\infty$.

\section{Notations and preliminaries}
A weight will always mean a non-negative function $\omega$ on $\mathbb R^d$ which is locally integrable. For a Lebesgue measurable set $E\subset\mathbb R^d$ and a weight $\omega$, we use the notation $|E|$ to denote the Lebesgue measure of $E$ and $\omega(E)$ to denote the weighted measure of $E$,
\begin{equation*}
\omega(E):=\int_E \omega(x)\,dx.
\end{equation*}
In the sequel, for any positive number $\gamma>0$, we denote $\omega^{\gamma}(x):=\omega(x)^{\gamma}$ by convention. For a measurable set $E\subset\mathbb R^d$, we let
\begin{equation*}
\omega^{\gamma}(E)=(\omega^{\gamma})(E):=\int_E \omega^{\gamma}(x)\,dx.
\end{equation*}
For any given ball $B=B(x_0,r)$ and $\lambda\in(0,\infty)$, we will write $\lambda B$ for the $\lambda$-dilate ball, which is the ball with the same center $x_0$ and radius $\lambda r$; that is $\lambda B=B(x_0,\lambda r)$. Similarly, $Q(x_0,r)$ denotes the cube centered at $x_0$ and with the sidelength $r$. Here and in what follows, only cubes with sides parallel to the coordinate axes are considered, and $\lambda Q=Q(x_0,\lambda r)$.
Let us recall two classes of weights that are given in terms of the critical radius function \eqref{rho}. As in \cite{bong1} (see also \cite{bong2}), we say that a weight $\omega$ belongs to the class $A^{\rho,\theta}_p(\mathbb R^d)$ for $1<p<\infty$ and $0<\theta<\infty$, if there is a positive constant $C>0$ such that for all balls $B=B(x_0,r)\subset\mathbb R^d$ with $x_0\in\mathbb R^d$ and $r\in(0,\infty)$,
\begin{equation*}
\bigg(\frac{1}{|B|}\int_B \omega(x)\,dx\bigg)^{1/p}\bigg(\frac{1}{|B|}\int_B \omega(x)^{-{p'}/p}\,dx\bigg)^{1/{p'}}
\leq C\cdot\bigg(1+\frac{r}{\rho(x_0)}\bigg)^{\theta},
\end{equation*}
where $p'=p/{(p-1)}$ denotes the \emph{conjugate exponent} of $p$, namely, $1/p+1/{p'}=1$. For $p=1$ and $0<\theta<\infty$, we also say that a weight $\omega$ belongs to the class $A^{\rho,\theta}_1(\mathbb R^d)$, if there is a positive constant $C>0$ such that for all balls $B=B(x_0,r)\subset\mathbb R^d$ with $x_0\in\mathbb R^d$ and $r\in(0,\infty)$,
\begin{equation*}
\frac1{|B|}\int_B \omega(x)\,dx\leq C\cdot\bigg(1+\frac{r}{\rho(x_0)}\bigg)^{\theta}\underset{x\in B}{\mbox{ess\,inf}}\;\omega(x).
\end{equation*}
Since
\begin{equation}\label{cc}
1\leq\bigg(1+\frac{r}{\rho(x_0)}\bigg)^{\theta_1}\leq\bigg(1+\frac{r}{\rho(x_0)}\bigg)^{\theta_2}
\end{equation}
whenever $0<\theta_1<\theta_2<\infty$, then for given $p$ with $1\leq p<\infty$, by definition, we have
\begin{equation*}
A_p(\mathbb R^d)\subset A^{\rho,\theta_1}_p(\mathbb R^d)\subset A^{\rho,\theta_2}_p(\mathbb R^d),
\end{equation*}
where $A_p(\mathbb R^d)$ denotes the classical Muckenhoupt class (see \cite{muck} and \cite{coi}). For any given $1\leq p<\infty$, as the classes $A^{\rho,\theta}_p(\mathbb R^d)$ increase with respect to $\theta$, it is natural to define
\begin{equation*}
A^{\rho,\infty}_p(\mathbb R^d):=\bigcup_{\theta>0}A^{\rho,\theta}_p(\mathbb R^d).
\end{equation*}
Consequently, one has the inclusion relation
\begin{equation*}
A_p(\mathbb R^d)\subset A^{\rho,\infty}_p(\mathbb R^d),\quad 1\leq p<\infty.
\end{equation*}
However, the converse is not true, it is easy to check that the above inclusion is strict. In fact, if $\omega\in A_p(\mathbb R^d)$ for some $p\geq1$, then $\omega(x)\,dx$ is a doubling measure (see \cite{duoand} and \cite{grafakos}), i.e., there exists a universal constant $C>0$ such that for any ball $B$ in $\mathbb R^d$,
\begin{equation*}
\omega(2B)\leq C\omega(B).
\end{equation*}
If $\omega\in A^{\rho,\theta}_p(\mathbb R^d)$ for some $p\geq1$ and $\theta>0$, then $\omega(x)\,dx$ may not be a doubling measure. For example, the weight
\begin{equation*}
\omega_{\gamma}(x)=(1+|x|)^{\gamma}\in A^{\rho,\theta}_p(\mathbb R^d)\quad \mbox{for any}~~ \gamma>d(p-1),
\end{equation*}
provided that $V\equiv1$ and $\rho(\cdot)\equiv1$. It is easy to see that such choice of $\omega_{\gamma}$ yields $\omega_{\gamma}(x)\,dx$ is not a doubling measure, hence it does not belong to $A_q(\mathbb R^d)$ for any $1\leq q<\infty$. The situation is more complicated. We can define (generalized) doubling classes of weights adapted to the Schr\"{o}dinger context, see \cite{bong9} and \cite{bong10}, for example.
In addition, for some fixed $\theta>0$, we have the following inclusion relations (see \cite{tang})
\begin{equation*}
A^{\rho,\theta}_1(\mathbb R^d)\subset A^{\rho,\theta}_{p_1}(\mathbb R^d)\subset A^{\rho,\theta}_{p_2}(\mathbb R^d),
\end{equation*}
whenever $1\leq p_1<p_2<\infty$.
As in the classical Muckenhoupt theory, we define the $A^{\rho,\theta}_{p}$ characteristic constants of $\omega$ as follows:
\begin{equation*}
[\omega]_{A^{\rho,\theta}_p}
:=\sup_{B\subset\mathbb R^d}\bigg(1+\frac{r}{\rho(x_0)}\bigg)^{-\theta}\bigg(\frac{1}{|B|}\int_B \omega(x)\,dx\bigg)^{1/p}\bigg(\frac{1}{|B|}\int_B \omega(x)^{-{p'}/p}\,dx\bigg)^{1/{p'}},
\end{equation*}
for $1<p<\infty,$ and
\begin{equation*}
[\omega]_{A^{\rho,\theta}_1}
:=\sup_{B\subset\mathbb R^d}\bigg(1+\frac{r}{\rho(x_0)}\bigg)^{-\theta}\bigg(\frac{1}{|B|}\int_B \omega(x)\,dx\bigg)
\bigg(\underset{x\in B}{\mbox{ess\,inf}}\,\omega(x)\bigg)^{-1},
\end{equation*}
for $p=1$, where the supremum is taken over all balls $B=B(x_0,r)$ in $\mathbb R^d$. In view of \eqref{cc}, we can see that if $\omega\in A^{\rho,\theta_1}_{p}(\mathbb R^d)$ with $1\leq p<\infty$ and $0\leq\theta_1<\infty$, then for any $\theta_1<\theta_2<\infty$, we have
\begin{equation*}
\omega\in A^{\rho,\theta_2}_{p}(\mathbb R^d)~~~~\qquad \&  ~~~~\qquad [\omega]_{A^{\rho,\theta_2}_p}\leq[\omega]_{A^{\rho,\theta_1}_p}.
\end{equation*}
Hence, for any given $\omega\in A^{\rho,\infty}_{p}(\mathbb R^d)$ with $1\leq p<\infty$, we let
\begin{equation*}
\theta^{\ast}:=\inf\Big\{\theta>0:\omega\in A^{\rho,\theta}_{p}(\mathbb R^d)\Big\}.
\end{equation*}
Now define the $A^{\rho,\infty}_{p}$ characteristic constant of $\omega$ by
\begin{equation*}
[\omega]_{A^{\rho,\infty}_p}:=[\omega]_{A^{\rho,\theta^{\ast}}_p}.
\end{equation*}
\begin{rem}
\begin{enumerate}
  \item It is well known that Muckenhoupt $A_p$ weights can be characterized by the weighted $L^p$ boundedness of the Hardy--Littlewood maximal operator and the Hilbert transform. For any given $\theta>0$, let us introduce the (Hardy--Littlewood type) maximal operator which is given in terms of the critical radius function \eqref{rho}.
\begin{equation*}
\mathcal{M}_{\rho,\theta}f(x):=\sup_{r>0}\bigg(1+\frac{r}{\rho(x)}\bigg)^{-\theta}\frac{1}{|B(x,r)|}\int_{B(x,r)}|f(y)|\,dy,\quad x\in\mathbb R^d.
\end{equation*}
The new classes $A^{\rho,\infty}_p(\mathbb R^d)$ are closely related to the family of maximal operators $\mathcal{M}_{\rho,\theta}$.
  \item Observe that a weight $\omega$ belongs to the class $A^{\rho,\infty}_1(\mathbb R^d)$ adapted to the Schr\"{o}dinger operator $\mathcal{L}$ if and only if there exists a positive number $\theta>0$ such that
      \begin{equation*}
      \mathcal{M}_{\rho,\theta}(\omega)(x)\leq C\cdot\omega(x),\quad \mbox{for}~a.e.~x\in\mathbb R^d,
      \end{equation*}
where the constant $C>0$ is independent of $\omega$. Moreover, as in the classical setting, the classes of weights adapted to the Schr\"{o}dinger operator $\mathcal{L}$ are characterized by the weighted $L^p$ boundedness of the corresponding maximal operators. Let $1<p<\infty$. It can be shown that $\omega\in A^{\rho,\infty}_p(\mathbb R^d)$ if and only if there exists a positive number $\theta>0$ such that $\mathcal{M}_{\rho,\theta}$ is bounded on $L^p(\omega)$ (see \cite{bong6} and \cite{bong7}, for example).
\item For the quantitative weighted estimates of the maximal operator $\mathcal{M}_{\rho,\theta}$ with $\theta>0$(the sharp $A^{\rho,\theta}_{p}$ bounds, which means the exponent of the characteristic constant $[\omega]_{A^{\rho,\theta}_p}$ is the best possible), see \cite{li} and \cite{zhang}.
\end{enumerate}
\end{rem}
As in \cite{tang} (see also \cite{bui} and \cite{wang}), we say that a weight $\omega$ is in the class
$A^{\rho,\theta}_{p,q}(\mathbb R^d)$ for $1<p,q<\infty$ and $0<\theta<\infty$, if there exists a positive constant $C>0$ such that
\begin{equation*}
\bigg(\frac{1}{|B|}\int_B \omega(x)^q\,dx\bigg)^{1/q}\bigg(\frac{1}{|B|}\int_B \omega(x)^{-{p'}}\,dx\bigg)^{1/{p'}}
\leq C\cdot\bigg(1+\frac{r}{\rho(x_0)}\bigg)^{\theta}
\end{equation*}
holds for all balls $B=B(x_0,r)\subset\mathbb R^d$. For the case $p=1$, we also say that a weight $\omega$ is in the class
$A^{\rho,\theta}_{1,q}(\mathbb R^d)$ for $1<q<\infty$ and $0<\theta<\infty$, if there exists a positive constant $C>0$ such that
\begin{equation*}
\bigg(\frac1{|B|}\int_B \omega(x)^q\,dx\bigg)^{1/q}\leq C\cdot\bigg(1+\frac{r}{\rho(x_0)}\bigg)^{\theta}
\underset{x\in B}{\mbox{ess\,inf}}\,\omega(x)
\end{equation*}
holds for all balls $B=B(x_0,r)\subset\mathbb R^d$. As before, we define the $A^{\rho,\theta}_{p,q}$ characteristic constants of $\omega$ as follows:
\begin{equation*}
[\omega]_{A^{\rho,\theta}_{p,q}}
:=\sup_{B\subset\mathbb R^d}\bigg(1+\frac{r}{\rho(x_0)}\bigg)^{-\theta}\bigg(\frac{1}{|B|}\int_B \omega(x)^q\,dx\bigg)^{1/q}\bigg(\frac{1}{|B|}\int_B \omega(x)^{-{p'}}\,dx\bigg)^{1/{p'}},
\end{equation*}
for $1<p,q<\infty$, and
\begin{equation*}
[\omega]_{A^{\rho,\theta}_{1,q}}
:=\sup_{B\subset\mathbb R^d}\bigg(1+\frac{r}{\rho(x_0)}\bigg)^{-\theta}\bigg(\frac1{|B|}\int_B \omega(x)^q\,dx\bigg)^{1/q}
\bigg(\underset{x\in B}{\mbox{ess\,inf}}\,\omega(x)\bigg)^{-1},
\end{equation*}
for $p=1$ and $1<q<\infty$. In view of \eqref{cc}, for any $1\leq p,q<\infty$, we find that
\begin{equation*}
A_{p,q}(\mathbb R^d)\subset A^{\rho,\theta_1}_{p,q}(\mathbb R^d)\subset A^{\rho,\theta_2}_{p,q}(\mathbb R^d),
\end{equation*}
whenever $0\leq\theta_1<\theta_2<\infty$, and
\begin{equation*}
[\omega]_{A^{\rho,\theta_2}_{p,q}}\leq[\omega]_{A^{\rho,\theta_1}_{p,q}}.
\end{equation*}
Here $A_{p,q}(\mathbb R^d)$ denotes the classical Muckenhoupt--Wheeden class (see \cite{muckenhoupt3}). Correspondingly, for $1\leq p,q<\infty$, we define
\begin{equation*}
A^{\rho,\infty}_{p,q}(\mathbb R^d):=\bigcup_{\theta>0}A^{\rho,\theta}_{p,q}(\mathbb R^d).
\end{equation*}
Obviously, for any fixed $\theta>0$,
\begin{equation*}
A_{p,q}(\mathbb R^d)\subset A^{\rho,\theta}_{p,q}(\mathbb R^d)\subset A^{\rho,\infty}_{p,q}(\mathbb R^d),\quad 1\leq p,q<\infty.
\end{equation*}
As before, for any given $\omega\in A^{\rho,\infty}_{p,q}(\mathbb R^d)$ with $1\leq p,q<\infty$, we let
\begin{equation*}
\theta^{\ast\ast}:=\inf\Big\{\theta>0:\omega\in A^{\rho,\theta}_{p,q}(\mathbb R^d)\Big\}.
\end{equation*}
Now define the $A^{\rho,\infty}_{p,q}$ characteristic constant of $\omega$ by
\begin{equation*}
[\omega]_{A^{\rho,\infty}_{p,q}}:=[\omega]_{A^{\rho,\theta^{\ast\ast}}_{p,q}}.
\end{equation*}
\begin{rem}
A few additional remarks are in order.
\begin{enumerate}
  \item Let us mention that in the definitions of both classes of weights $A^{\rho,\infty}_{p}(\mathbb R^d)$ and $A^{\rho,\infty}_{p,q}(\mathbb R^d)$, we can replace a ball $B(x,r)$ by a cube $Q(x,t)$ centered at $x$ with side length $t$, due to \eqref{com}.
  \item For more results about weighted norm inequalities of various integral operators in harmonic analysis (such as first or second order Riesz--Schr\"{o}dinger transforms, Schr\"{o}dinger type singular integrals, fractional integrals associated to the Schr\"{o}dinger operator $\mathcal{L}$, etc.), one can see \cite{bong1,bong2,bong3,bong4,bong5,bong6,bong7,bong8,bong9,bong10,bui,tang,tang2,wang,wang2}. Moreover, the quantitative weighted estimates for some operators such as the fractional maximal and integral operators, and Littlewood--Paley functions were recently obtained in \cite{li} and \cite{zhang}.
\end{enumerate}
\end{rem}
Throughout this paper, the letter $C$ stands for a positive constant which may vary from line to line. Constants with subscripts, such as $C_0$, do not change in different occurrences.

\section{Technical lemmas}
In this section, let us first set up two auxiliary lemmas. The auxiliary function $\rho(x)$ has many useful properties, the fundamental one is listed in Lemma \ref{N0}. This result implies, in particular, that when $x\in B(x_0,r)$ with $x_0\in\mathbb R^d$ and $r\in(0,\infty)$,
\begin{equation}\label{wangh3}
1+\frac{r}{\rho(x)}\leq C_0\cdot\bigg(1+\frac{r}{\rho(x_0)}\bigg)^{N_0+1},
\end{equation}
where $C_0\geq1$ is the constant appearing in \eqref{com}. In fact, this estimate has been obtained in the literature (see
\cite[Lemma 1]{bong8} and \cite[Lemma 2]{bong10}), for the convenience of the reader, we give its proof here. By the left-hand side of \eqref{com}, we can see that for any $x\in B(x_0,r)$ with $x_0\in\mathbb R^d$ and $r\in(0,\infty)$,
\begin{equation*}
\frac{1}{\rho(x)}\leq C_0\cdot\frac{1}{\rho(x_0)}\bigg(1+\frac{|x-x_0|}{\rho(x_0)}\bigg)^{N_0}
<C_0\cdot\frac{1}{\rho(x_0)}\bigg(1+\frac{r}{\rho(x_0)}\bigg)^{N_0}.
\end{equation*}
From this, it follows that (since $C_0\geq1$)
\begin{equation*}
1+\frac{r}{\rho(x)}\leq 1+C_0\cdot\frac{r}{\rho(x_0)}\bigg(1+\frac{r}{\rho(x_0)}\bigg)^{N_0}
\leq C_0\cdot\bigg(1+\frac{r}{\rho(x_0)}\bigg)^{N_0+1},
\end{equation*}
as desired.

Similar to the classical Muckenhoupt weights, there are some basic properties for $A^{\rho,\theta}_p(\mathbb R^d)$ classes of weights. The following important property for $A^{\rho,\theta}_p$ weights with $1\leq p<\infty$ and $0<\theta<\infty$ was first given by Bongioanni--Harboure--Salinas in \cite[Lemma 5]{bong1}.
\begin{lem}[\cite{bong1}]\label{rh}
If $\omega\in A^{\rho,\theta}_p(\mathbb R^d)$ with $0<\theta<\infty$ and $1\leq p<\infty$, then there exist positive constants $\epsilon>0,\eta>1$ and $C>0$ such that
\begin{equation}\label{rholder}
\bigg(\frac{1}{|\mathcal{Q}|}\int_{\mathcal{Q}}\omega(x)^{1+\epsilon}dx\bigg)^{\frac{1}{1+\epsilon}}
\leq C\cdot\bigg(\frac{1}{|\mathcal{Q}|}\int_{\mathcal{Q}}\omega(x)\,dx\bigg)\bigg(1+\frac{r}{\rho(x_0)}\bigg)^{\eta}
\end{equation}
holds for every cube $\mathcal{Q}=Q(x_0,r)$ in $\mathbb R^d$.
\end{lem}
\begin{rem}
The constant $C>0$ in Lemma \ref{rh} depends on $p,d$ and the $A^{\rho,\theta}_{p}$ characteristic constant of $\omega$, the positive number $\epsilon$ in Lemma \ref{rh} comes from the classical proof for $A_p$ weights in \cite[Theorem 7.4]{duoand}, and $\eta$ is a positive constant greater than 1, which can be chosen as follows.
\begin{equation*}
\eta:=\theta p+(\theta+d)\frac{pN_0}{N_0+1}+(N_0+1)\frac{d\epsilon}{1+\epsilon}>1.
\end{equation*}
\end{rem}
As a direct consequence of Lemma \ref{rh}, we can prove the following result, which provides us the comparison between the Lebesgue measure of the subset $E$ of $\mathbb R^d$ and its weighted measure $\omega(E)$.
\begin{lem}\label{comparelem}
If $\omega\in A^{\rho,\theta}_p(\mathbb R^d)$ with $0<\theta<\infty$ and $1\leq p<\infty$, then there exist two positive numbers $0<\delta<1$ and $\eta>1$ such that for any cube $\mathcal{Q}=Q(x_0,r)$ in $\mathbb R^d$,
\begin{equation}\label{compare}
\frac{\omega(E)}{\omega(\mathcal{Q})}\leq C\cdot\bigg(\frac{|E|}{|\mathcal{Q}|}\bigg)^\delta\bigg(1+\frac{r}{\rho(x_0)}\bigg)^{\eta}
\end{equation}
holds for any measurable subset $E$ contained in $\mathcal{Q}$, where $C>0$ is a constant which does not depend on $E$ nor on $\mathcal{Q}$, and $\eta$ is given as in Lemma \ref{rh}.
\end{lem}
\begin{proof}
For any given cube $\mathcal{Q}=Q(x_0,r)$ with $x_0\in\mathbb R^d$ and $r\in(0,\infty)$, suppose that $E\subset \mathcal{Q}$, then by H\"older's inequality with exponent $1+\epsilon$ and \eqref{rholder}, we can deduce that
\begin{equation*}
\begin{split}
\omega(E)&=\int_{\mathcal{Q}}\chi_E(x)\cdot \omega(x)\,dx\\
&\leq\bigg(\int_{\mathcal{Q}} \omega(x)^{1+\epsilon}dx\bigg)^{\frac{1}{1+\epsilon}}
\bigg(\int_{\mathcal{Q}}\chi_E(x)^{\frac{1+\epsilon}{\epsilon}}\,dx\bigg)^{\frac{\epsilon}{1+\epsilon}}\\
&\leq C\cdot|\mathcal{Q}|^{\frac{1}{1+\epsilon}}\bigg(\frac{1}{|\mathcal{Q}|}\int_{\mathcal{Q}}\omega(x)\,dx\bigg)
\bigg(1+\frac{r}{\rho(x_0)}\bigg)^{\eta}|E|^{\frac{\epsilon}{1+\epsilon}}\\
&=C\cdot\bigg(\frac{|E|}{|\mathcal{Q}|}\bigg)^{\frac{\epsilon}{1+\epsilon}}\bigg(1+\frac{r}{\rho(x_0)}\bigg)^{\eta}\omega(\mathcal{Q}).
\end{split}
\end{equation*}
This gives \eqref{compare} with $\delta=\epsilon/{(1+\epsilon)}$. Here the characteristic function of the set $E$ is denoted by $\chi_E$.
\end{proof}

The following result gives the relationship between these two classes of weights, $A^{\rho,\infty}_p(\mathbb R^d)$ and $A^{\rho,\infty}_{p,q}(\mathbb R^d)$, which can be found in \cite{wang} and \cite{wang3}.
\begin{lem}[\cite{wang,wang3}]\label{Apq}
Suppose that $1\leq p<q<\infty$. Then the following statements are true.
\begin{enumerate}
  \item If $p>1$ and $0<\theta<\infty$, then $\omega\in A^{\rho,\theta}_{p,q}(\mathbb R^d)$ implies that $\omega^q\in A^{\rho,\widetilde{\theta}}_t(\mathbb R^d)$ with
\begin{equation*}
t:=1+q/{p'}\quad and \quad \widetilde{\theta}:=\theta\cdot\frac{1}{1/q+1/{p'}}.
\end{equation*}
  \item If $p=1$ and $0<\theta<\infty$, then $\omega\in A^{\rho,\theta}_{1,q}(\mathbb R^d)$ implies that $\omega^q\in A^{\rho,\theta^{\ast}}_1(\mathbb R^d)$ with
  \begin{equation*}
  \theta^{\ast}:=\theta\cdot q.
  \end{equation*}
\end{enumerate}
\end{lem}

\section{Known results}
In this section, we will present some relevant results concerning characterizations of several function spaces in the literature. As we mentioned in the introduction, the celebrated John--Nirenberg inequality states that if $f\in\mathrm{BMO}(\mathbb R^d)$, then for any cube $Q$ in $\mathbb R^d$ and for any $\lambda>0$,
\begin{equation*}
\Big|\Big\{x\in Q:|f(x)-f_{Q}|>\lambda\Big\}\Big|
\leq C_1|Q|\exp\bigg\{-\frac{C_2\lambda}{\|f\|_{\mathrm{BMO}}}\bigg\},
\end{equation*}
where $C_1>0$ and $C_2>0$ are two universal constants (see \cite{john} and \cite{duoand}). The John--Nirenberg inequality shows that any BMO function is exponentially integrable. As a consequence of this estimate and H\"{o}lder's inequality, we can obtain an equivalent norm on $\mathrm{BMO}(\mathbb R^d)$, see \cite[Corollary 6.12]{duoand}, for example.
\begin{prop}[\cite{duoand}]
For $1\leq s<\infty$, define
\begin{equation*}
\|f\|_{\mathrm{BMO}^s}:=\sup_{Q\subset\mathbb R^d}\bigg(\frac{1}{|Q|}\int_{Q}|f(x)-f_{Q}|^s\,dx\bigg)^{1/s}.
\end{equation*}
Then we have~(when $s=1$, we write $\|\cdot\|_{\mathrm{BMO}^s}=\|\cdot\|_{\mathrm{BMO}}$)
\begin{equation*}
\|f\|_{\mathrm{BMO}^s}\approx\|f\|_{\mathrm{BMO}},
\end{equation*}
for each $1<s<\infty$.
\end{prop}
Now we define
\begin{equation*}
\mathrm{BMO}^s(\mathbb R^d):=\Big\{f\in L^1_{\mathrm{loc}}(\mathbb R^d):\|f\|_{\mathrm{BMO}^s}<\infty\Big\},\quad 1\leq s<\infty.
\end{equation*}
This result tells us that for all $1\leq s<\infty$, the spaces $\mathrm{BMO}^s(\mathbb R^d)$ coincide, and the norms $\|\cdot\|_{\mathrm{BMO}^s}$ are mutually equivalent with respect to different values of $s$.

We can extend this result to the weighted case. For each $\omega\in A_{\infty}(\mathbb R^d):=\cup_{1\leq p<\infty}A_p(\mathbb R^d)$, we denote by $\mathrm{BMO}({\omega})$ the set of all locally integrable functions $f$ on $\mathbb R^d$ such that
\begin{equation*}
\|f\|_{\mathrm{BMO}(\omega)}:=\sup_{Q\subset\mathbb R^d}\frac{1}{\omega(Q)}\int_{Q}|f(x)-f_{\omega,Q}|\omega(x)\,dx<\infty,
\end{equation*}
where
\begin{equation*}
f_{\omega,Q}:=\frac{1}{\omega(Q)}\int_{Q}f(x)\omega(x)\,dx.
\end{equation*}

In 1976, Muckenhoupt and Wheeden proved that a function $f$ is in the space $\mathrm{BMO}(\mathbb R^d)$ if and only if $f$ is in the space $\mathrm{BMO}({\omega})$ (bounded mean oscillation with respect to $\omega$), provided that $\omega\in A_{\infty}(\mathbb R^d)$, one can see \cite[Theorem 5]{muck2}.
\begin{prop}[\cite{muck2}]
For each $\omega\in A_{\infty}(\mathbb R^d)$, then we have $\mathrm{BMO}(\mathbb R^d)=\mathrm{BMO}({\omega})$ and (the norms are mutually equivalent)
\begin{equation*}
\|f\|_{\mathrm{BMO}(\omega)}\approx\|f\|_{\mathrm{BMO}}.
\end{equation*}
\end{prop}

In 2011, Ho further proved the following result by using H\"older's inequality, the John--Nirenberg inequality and relevant properties of $A_{\infty}$ weights, see \cite[Theorem 3.1]{ho}.
\begin{prop}[\cite{ho}]\label{hojian}
For all $1\leq s<\infty$ and $\omega\in A_{s}(\mathbb R^d)$, then $f\in \mathrm{BMO}(\mathbb R^d)$ if and only if
\begin{equation*}
\sup_{Q\subset\mathbb R^d}\bigg(\frac{1}{\omega(Q)}\int_{Q}|f(x)-f_{Q}|^s\omega(x)\,dx\bigg)^{1/s}<\infty.
\end{equation*}
\end{prop}
By using similar arguments, we can prove a version of John--Nirenberg inequality suitable for the BLO spaces
(see, for instance, \cite[Lemma 2.1]{wangding}).
\begin{equation}\label{bloineq}
\Big|\Big\{x\in \mathcal{Q}:\Big[f(x)-\underset{y\in\mathcal{Q}}{\mathrm{ess\,inf}}\,f(y)\Big]>\lambda\Big\}\Big|\leq
C_3|\mathcal{Q}|\exp\bigg\{-\frac{C_4\lambda}{\|f\|_{\mathrm{BLO}}}\bigg\}.
\end{equation}
Here $C_3$ and $C_4$ are two absolute constants. There is an analogue of Proposition \ref{hojian} for the space $\mathrm{BLO}(\mathbb R^d)$. Based on the estimate \eqref{bloineq}, we further obtain the following result.
\begin{prop}
For all $1\leq s<\infty$ and $\omega\in A_{s}(\mathbb R^d)$, then $f\in \mathrm{BLO}(\mathbb R^d)$ if and only if
\begin{equation*}
\sup_{Q\subset\mathbb R^d}\bigg(\frac{1}{\omega(Q)}\int_{Q}\Big[f(x)-\underset{y\in\mathcal{Q}}{\mathrm{ess\,inf}}\,f(y)\Big]^s\omega(x)\,dx\bigg)^{1/s}<\infty.
\end{equation*}
\end{prop}
This result was first given by Wang--Zhou--Teng in 2018 (to the best of our knowledge), see \cite[Theorem 3.1]{wangding}.

Moreover, there are many works about the characterizations of classical BMO and BLO spaces, one can see \cite{coi2,coifman,kom1,lu,shirai} and the references therein. For the boundedness properties of some operators in BMO and BLO spaces, see \cite{meng1,meng2,ou}. On the other hand, we have the following characterization of classical Campanato spaces, which can be found in \cite[Lemma 1.5]{pal} and \cite[Theorem 2]{janson}.
\begin{prop}[\cite{janson,pal}]\label{camwangchen}
For $1\leq s<\infty$ and $0<\beta<1$, define
\begin{equation*}
\|f\|_{{\mathcal{C}}^{\beta,s}}:=\sup_{B\subset\mathbb R^d}\frac{1}{|B|^{\beta/d}}\bigg(\frac{1}{|B|}\int_{B}|f(x)-f_{B}|^s\,dx\bigg)^{1/s},
\end{equation*}
and
\begin{equation*}
\|f\|_{\mathrm{Lip}_{\beta}}:=\sup_{x,y\in \mathbb R^d,x\neq y}\frac{|f(x)-f(y)|}{|x-y|^{\beta}}.
\end{equation*}
Then we have (when $s=1$, we denote $\|\cdot\|_{\mathcal{C}^{\beta,s}}=\|\cdot\|_{\mathcal{C}^{\beta}}$)
\begin{equation*}
\|f\|_{\mathcal{C}^{\beta,s}}\approx\|f\|_{\mathcal{C}^{\beta}}\approx\|f\|_{\mathrm{Lip}_{\beta}},
\end{equation*}
for each $1<s<\infty$.
\end{prop}
As before, we also define
\begin{equation*}
\mathcal{C}^{\beta,s}(\mathbb R^d):=\Big\{f\in L^1_{\mathrm{loc}}(\mathbb R^d):\|f\|_{{\mathcal{C}}^{\beta,s}}<\infty\Big\}.
\end{equation*}
Proposition \ref{camwangchen} now tells us that for all $1\leq s<\infty$, the spaces $\mathcal{C}^{\beta,s}(\mathbb R^d)$ coincide, and the norms $\|\cdot\|_{{\mathcal{C}}^{\beta,s}}$ are equivalent with respect to different values of $s$.
\begin{rem}
\begin{enumerate}
  \item We mention that this result leads to a generalization of the classical Sobolev embedding theorem. It is also well known that $\mathrm{Lip}_{1/p-1}(\mathbb R^d)$ is the dual space of Hardy space $H^p(\mathbb R^d)$ when $0<p<1$, and $\mathrm{Lip}_{0}(\mathbb R^d)=\mathrm{BMO}(\mathbb R^d)$ is the dual space of Hardy space $H^1(\mathbb R^d)$, see \cite{duoand,grafakos,grafakos2}.
  \item There are some other characterizations of Campanato and Lipschitz spaces, which have been obtained by several authors. For instance, we can give some new characterizations of Campanato and Lipschitz spaces via the boundedness of commutators (such as Calder\'{o}n--Zygmund singular integral operators and fractional integrals). We can also obtain Littlewood--Paley characterizations of Lipschitz spaces using the Littlewood--Paley theory. For further details, we refer the reader to \cite{deng,duong,grafakos,pal,shi,shi2} and the references therein.
  \item For the weighted version of Campanato spaces, see \cite{tang3} and \cite{yang} for example. For more general results in the context of spaces of homogeneous type, see \cite{nakai} and \cite{tang4} for example.
\end{enumerate}
\end{rem}
It is natural to consider the same problems (characterizations of function spaces) in the Schr\"{o}dinger context. Concerning the BMO and Campanato spaces related to Schr\"{o}dinger operators with nonnegative potentials, we can obtain the following conclusions.
\begin{prop}
Let $0<\theta<\infty$ and $1\leq s<\infty$. If $f\in \mathrm{BMO}_{\rho,\theta}(\mathbb R^d)$, then there exists a positive constant $C>0$ such that
\begin{equation*}
\bigg(\frac{1}{|\mathcal{Q}|}\int_{\mathcal{Q}}|f(x)-f_{\mathcal{Q}}|^s\,dx\bigg)^{1/s}
\leq C\bigg(1+\frac{r}{\rho(x_0)}\bigg)^{(N_0+1)\theta}\|f\|_{\mathrm{BMO}_{\rho,\theta}}
\end{equation*}
holds for every cube $\mathcal{Q}=Q(x_0,r)$ with $x_0\in\mathbb R^d$ and $r>0$, where $N_0$ is the constant appearing in Lemma $\ref{N0}$.
\end{prop}
This result was first proved by Bongioanni--Harboure--Salinas in 2011, see \cite[Proposition 3]{bong3}.

\begin{lem}[\cite{liu}]\label{beta}
If $f\in\mathcal{C}^{\beta}_{\rho,\theta}(\mathbb R^d)$ with $0<\beta<1$ and $0<\theta<\infty$, then there exists a positive constant $C>0$ such that
\begin{equation*}
\frac{|f(x)-f(y)|}{|x-y|^{\beta}}\leq C\|f\|_{\mathcal{C}^{\beta}_{\rho,\theta}}
\bigg(1+\frac{|x-y|}{\rho(x)}+\frac{|x-y|}{\rho(y)}\bigg)^{\theta}
\end{equation*}
holds true for all $x,y\in\mathbb R^d$ with $x\neq y$. Conversely, if there is a positive constant $C>0$ such that for any $x,y\in\mathbb R^d$ with $x\neq y$,
\begin{equation*}
\frac{|f(x)-f(y)|}{|x-y|^{\beta}}\leq C\bigg(1+\frac{|x-y|}{\rho(x)}+\frac{|x-y|}{\rho(y)}\bigg)^{\theta}
\end{equation*}
holds for some $\theta>0$ and $0<\beta<1$, then $f\in \mathcal{C}^{\beta}_{\rho,(N_0+1)\theta}(\mathbb R^d)$.
\end{lem}

\begin{prop}
Let $0<\theta<\infty$ and $1\leq s<\infty$. If $f\in \mathcal{C}^{\beta}_{\rho,\theta}(\mathbb R^d)$ with $0<\beta<1$, then there exists a positive constant $C>0$ such that
\begin{equation*}
\frac{1}{|{\mathcal{B}}|^{\beta/d}}\bigg(\frac{1}{|\mathcal{B}|}\int_{\mathcal{B}}|f(x)-f_{\mathcal{B}}|^s\,dx\bigg)^{1/s}
\leq C\bigg(1+\frac{r}{\rho(x_0)}\bigg)^{(N_0+1)\theta}\|f\|_{\mathrm{Lip}_{\beta}^{\rho,\theta}}
\end{equation*}
holds for every ball $\mathcal{B}=B(x_0,r)$ with $x_0\in\mathbb R^d$ and $r>0$, where $N_0$ is the constant appearing in Lemma $\ref{N0}$.
\end{prop}
This result was first given by Liu--Sheng in 2014, see \cite[Proposition 3]{liu}.

From the above overview, we can see that there are many problems to be studied in the new spaces $\mathrm{BLO}_{\rho,\theta}(\mathbb R^d)$ and $\mathcal{C}^{\beta,\ast}_{\rho,\theta}(\mathbb R^d)$. One is naturally led to ask whether it is possible to obtain a variant of the John--Nirenberg inequality for the space $\mathrm{BLO}_{\rho,\theta}(\mathbb R^d)$. In this paper we give a positive answer to this problem. Moreover, we give several results about characterizations for BLO space and Campanato space related to the Schr\"{o}dinger operator $\mathcal{L}=-\Delta+V$. This is a continuation of the previous work by the authors in \cite{wang3}.

As already mentioned in the introduction, the harmonic analysis arising from the Schr\"{o}dinger operator $\mathcal{L}=-\Delta+V$ is based on the use of a related critical radius function, which was introduced by Shen in \cite{shen}. In this framework, to show our main results, we rely on a version of the John--Nirenberg inequality for the space $\mathrm{BLO}_{\rho,\theta}(\mathbb R^d)$(see Lemma \ref{expblo} below), a pointwise estimate for the function $f\in \mathcal{C}^{\beta}_{\rho,\theta}(\mathbb R^d)$(see Lemma \ref{beta} above), and some related properties of classes of weights adapted to the Schr\"{o}dinger operator $\mathcal{L}$.

\section{John--Nirenberg type inequalities for the new spaces}
In this section, we are concerned with the John--Nirenberg type inequality with precise constants suitable for the $\mathrm{BLO}_{\rho,\theta}$ spaces and relevant properties.

\begin{lem}\label{expblo}
If $f\in \mathrm{BLO}_{\rho,\theta}(\mathbb R^d)$ with $0<\theta<\infty$, then there exist two positive constants $C_1$ and $C_2$ such that for every cube $\mathcal{Q}=Q(x_0,r)$ and every $\lambda>0$,
\begin{equation}\label{maincw}
\begin{split}
&\Big|\Big\{x\in \mathcal{Q}:\Big[f(x)-\underset{y\in\mathcal{Q}}{\mathrm{ess\,inf}}\,f(y)\Big]>\lambda\Big\}\Big|\\
&\leq \overline{C}_1|\mathcal{Q}|\exp\bigg\{-\bigg(1+\frac{r}{\rho(x_0)}\bigg)^{-(N_0+1)\theta}
\frac{\overline{C}_2\lambda}{\|f\|_{\mathrm{BLO}_{\rho,\theta}}}\bigg\},
\end{split}
\end{equation}
where $N_0$ is the constant appearing in Lemma \ref{N0}. More specifically, we may choose
\begin{equation*}
\overline{C}_1=e\quad and \quad\overline{C}_2=\frac{1}{C_0^{\theta}2^de}.
\end{equation*}
\end{lem}
\begin{proof}
Some ideas of the proof of this lemma come from \cite{duoand} and \cite{grafakos2}. The proof has five main steps.

\textbf{Step 1}. Without loss of generality, we may assume that $\|f\|_{\mathrm{BLO}_{\rho,\theta}}=1$ with $0<\theta<\infty$. Note that
\begin{equation*}
\bigg(1+\frac{r}{\rho(x_0)}\bigg)^{\theta}\geq1,\;~~\mbox{for any}~\theta>0.
\end{equation*}
If $\lambda\leq 1$, then the inequality \eqref{maincw} holds true by choosing $\overline{C}_1=e$ and $\overline{C}_2=1$. Now we suppose that $\lambda>1$. Then for each fixed cube $\mathcal{Q}=Q(x_0,r)$, we can apply the Calder\'{o}n--Zygmund decomposition to the function $f(x)-\underset{y\in\mathcal{Q}}{\mathrm{ess\,inf}}f(y)$ inside the cube $\mathcal{Q}$. Let $\sigma>1$ be a positive constant to be fixed below.
Since
\begin{equation*}
\bigg(1+\frac{r}{\rho(x_0)}\bigg)^{-\theta}
\bigg(\frac{1}{|\mathcal{Q}|}\int_{\mathcal{Q}}\Big[f(x)-\underset{y\in\mathcal{Q}}{\mathrm{ess\,inf}}\,f(y)\Big]\,dx\bigg)
\leq\|f\|_{\mathrm{BLO}_{\rho,\theta}}=1<\sigma,
\end{equation*}
we then follow the same argument(the so-called stopping time argument) as in the proof of \cite[Theorem 7.1.6]{grafakos2} to obtain a collection of (pairwise disjoint) cubes $\{Q^{(1)}_j\}_j$ satisfying the following properties:
\begin{equation*}
\begin{split}
&(A)\mbox{-1}. ~~\mbox{The interior of every cube}~ Q^{(1)}_j ~\mbox{is contained in}~ \mathcal{Q};\\
&(B)\mbox{-1}. ~~\sigma\bigg(1+\frac{r}{\rho(x_0)}\bigg)^{\theta}
<\frac{1}{|Q^{(1)}_j|}\int_{Q^{(1)}_j}\Big[f(x)-\underset{y\in\mathcal{Q}}{\mathrm{ess\,inf}}\,f(y)\Big]\,dx
\leq 2^d\sigma\bigg(1+\frac{r}{\rho(x_0)}\bigg)^{\theta};\\
&(C)\mbox{-1}. ~~0\leq\underset{y\in{Q^{(1)}_j}}{\mathrm{ess\,inf}}\,f(y)-\underset{y\in\mathcal{Q}}{\mathrm{ess\,inf}}\,f(y)
\leq 2^d\sigma\bigg(1+\frac{r}{\rho(x_0)}\bigg)^{\theta};\\
&(D)\mbox{-1}. ~~\sum_{j}\big|Q^{(1)}_j\big|\leq\frac{|\mathcal{Q}|}{\sigma};\\
&(E)\mbox{-1}. ~~f(x)-\underset{y\in\mathcal{Q}}{\mathrm{ess\,inf}}\,f(y)
\leq\sigma\bigg(1+\frac{r}{\rho(x_0)}\bigg)^{\theta},~~a.e.~x\in \mathcal{Q}\setminus\bigcup_jQ^{(1)}_j.
\end{split}
\end{equation*}
We prove these properties $(A)$-1 through $(E)$-1. Obviously, properties $(A)$-1 and $(B)$-1 hold by the selection criterion of the cubes $Q^{(1)}_j$(viewed as the first generation of $\mathcal{Q}$). Since $Q^{(1)}_j\subset \mathcal{Q}$ and
\begin{equation*}
\begin{split}
\underset{y\in{Q^{(1)}_j}}{\mathrm{ess\,inf}}\,f(y)
&=\frac{1}{|Q^{(1)}_j|}\int_{Q^{(1)}_j}\underset{y\in{Q^{(1)}_j}}{\mathrm{ess\,inf}}\,f(y)\,dx\\
&\leq\frac{1}{|Q^{(1)}_j|}\int_{Q^{(1)}_j}f(x)\,dx=f_{Q^{(1)}_j},
\end{split}
\end{equation*}
we get
\begin{equation*}
\begin{split}
0&\leq\underset{y\in{Q^{(1)}_j}}{\mathrm{ess\,inf}}\,f(y)-\underset{y\in\mathcal{Q}}{\mathrm{ess\,inf}}\,f(y)\\
&\leq\frac{1}{|Q^{(1)}_j|}\int_{Q^{(1)}_j}f(x)\,dx-\underset{y\in\mathcal{Q}}{\mathrm{ess\,inf}}\,f(y)\\
&=\frac{1}{|Q^{(1)}_j|}\int_{Q^{(1)}_j}\Big[f(x)-\underset{y\in\mathcal{Q}}{\mathrm{ess\,inf}}\,f(y)\Big]\,dx\\
&\leq 2^d\sigma\bigg(1+\frac{r}{\rho(x_0)}\bigg)^{\theta},
\end{split}
\end{equation*}
where in the last inequality we have used $(B)$-1. Because the cubes $Q^{(1)}_j$ are pairwise disjoint, then it follows from $(B)$-1 that
\begin{equation*}
\begin{split}
\bigg(1+\frac{r}{\rho(x_0)}\bigg)^{\theta}\sum_{j}\big|Q^{(1)}_j\big|&<\frac{1}{\sigma}
\sum_{j}\int_{Q^{(1)}_j}\Big[f(x)-\underset{y\in\mathcal{Q}}{\mathrm{ess\,inf}}\,f(y)\Big]\,dx\\
&=\frac{1}{\sigma}\int_{\bigcup_jQ^{(1)}_j}\Big[f(x)-\underset{y\in\mathcal{Q}}{\mathrm{ess\,inf}}\,f(y)\Big]\,dx\\
&\leq\frac{1}{\sigma}\int_{\mathcal{Q}}\Big[f(x)-\underset{y\in\mathcal{Q}}{\mathrm{ess\,inf}}\,f(y)\Big]\,dx\\
&\leq\frac{|\mathcal{Q}|}{\sigma}\bigg(1+\frac{r}{\rho(x_0)}\bigg)^{\theta}.
\end{split}
\end{equation*}
This is equivalent to $(D)$-1. $(E)$-1 is a consequence of the Lebesgue differentiation theorem.

\textbf{Step 2}. We now fix a selected cube $Q^{(1)}_{j'}$(first generation) and apply the same Calder\'{o}n--Zygmund decomposition to the function $f(x)-\underset{y\in{Q}^{(1)}_{j'}}{\mathrm{ess\,inf}}f(y)$ inside the cube $Q^{(1)}_{j'}$. Also repeat this process for any other cube of the first generation. Let $Q^{(1)}_{j'}=Q(x_1,r_1)$ be the cube centered at $x_1$ and with side length $r_1$. Observe that
\begin{equation*}
\bigg(1+\frac{r_1}{\rho(x_1)}\bigg)^{-\theta}
\bigg(\frac{1}{|Q^{(1)}_{j'}|}\int_{Q^{(1)}_{j'}}\Big[f(x)-\underset{y\in Q^{(1)}_{j'}}{\mathrm{ess\,inf}}\,f(y)\Big]\,dx\bigg)
\leq\|f\|_{\mathrm{BLO}_{\rho,\theta}}=1<\sigma.
\end{equation*}
Arguing as in Step 1, we obtain a collection of (pairwise disjoint) cubes $\{Q^{(2)}_j\}_j$ satisfying the following properties:
\begin{equation*}
\begin{split}
&(A)\mbox{-2}. ~~\mbox{The interior of every cube}~ Q^{(2)}_j ~\mbox{is contained in a unique cube}~ {Q}^{(1)}_{j'};\\
&(B)\mbox{-2}. ~~\sigma\bigg(1+\frac{r_1}{\rho(x_1)}\bigg)^{\theta}
<\frac{1}{|Q^{(2)}_j|}\int_{Q^{(2)}_j}\Big[f(x)-\underset{y\in{Q}^{(1)}_{j'}}{\mathrm{ess\,inf}}\,f(y)\Big]\,dx
\leq 2^d\sigma\bigg(1+\frac{r_1}{\rho(x_1)}\bigg)^{\theta};\\
&(C)\mbox{-2}. ~~0\leq\underset{y\in{Q}^{(2)}_{j}}{\mathrm{ess\,inf}}\,f(y)-\underset{y\in{Q^{(1)}_{j'}}}{\mathrm{ess\,inf}}\,f(y)
\leq 2^d\sigma\bigg(1+\frac{r_1}{\rho(x_1)}\bigg)^{\theta};\\
&(D)\mbox{-2}. ~~\sum_{j}\big|Q^{(2)}_j\big|\leq\frac{1}{\sigma}\sum_{j'}\big|Q^{(1)}_{j'}\big|;\\
&(E)\mbox{-2}. ~~f(x)-\underset{y\in{Q}^{(1)}_{j'}}{\mathrm{ess\,inf}}\,f(y)
\leq\sigma\bigg(1+\frac{r_1}{\rho(x_1)}\bigg)^{\theta},~~a.e.~x\in {Q}^{(1)}_{j'}\setminus\bigcup_jQ^{(2)}_j.
\end{split}
\end{equation*}
In fact, it is clear that properties $(A)$-2 and $(B)$-2 hold by the selection criterion of the cubes $Q^{(2)}_j$(viewed as the second generation of $\mathcal{Q}$). Since $Q^{(2)}_j\subset Q^{(1)}_{j'}$ and
\begin{equation*}
\begin{split}
\underset{y\in{Q^{(2)}_j}}{\mathrm{ess\,inf}}\,f(y)
&=\frac{1}{|Q^{(2)}_j|}\int_{Q^{(2)}_j}\underset{y\in{Q^{(2)}_j}}{\mathrm{ess\,inf}}\,f(y)\,dx
\leq\frac{1}{|Q^{(2)}_j|}\int_{Q^{(2)}_j}f(x)\,dx,
\end{split}
\end{equation*}
so we have
\begin{equation*}
\begin{split}
0\leq\underset{y\in{Q}^{(2)}_{j}}{\mathrm{ess\,inf}}\,f(y)-\underset{y\in{Q^{(1)}_{j'}}}{\mathrm{ess\,inf}}\,f(y)
&\leq\frac{1}{|Q^{(2)}_j|}\int_{Q^{(2)}_j}\Big[f(x)-\underset{y\in{Q}^{(1)}_{j'}}{\mathrm{ess\,inf}}\,f(y)\Big]\,dx\\
&\leq 2^d\sigma\bigg(1+\frac{r_1}{\rho(x_1)}\bigg)^{\theta},
\end{split}
\end{equation*}
due to property $(B)$-2. By the Lebesgue differentiation theorem, $(E)$-2 holds. It remains only to study the last property $(D)$-2. Notice that the cubes $Q^{(2)}_j$ are also pairwise disjoint and each selected cube $Q^{(2)}_j$ is contained in a unique cube $Q^{(1)}_{j'}$, we can deduce that
\begin{equation*}
\begin{split}
\bigg(1+\frac{r_1}{\rho(x_1)}\bigg)^{\theta}\sum_{j}\big|Q^{(2)}_j\big|
&<\frac{1}{\sigma}\sum_{j}\int_{Q^{(2)}_j}\Big[f(x)-\underset{y\in{Q}^{(1)}_{j'}}{\mathrm{ess\,inf}}\,f(y)\Big]\,dx\\
&\leq\frac{1}{\sigma}\sum_{j'}\int_{Q^{(1)}_{j'}}\Big[f(x)-\underset{y\in{Q}^{(1)}_{j'}}{\mathrm{ess\,inf}}\,f(y)\Big]\,dx\\
&\leq\frac{1}{\sigma}\sum_{j'}\big|Q^{(1)}_{j'}\big|\bigg(1+\frac{r_1}{\rho(x_1)}\bigg)^{\theta}\|f\|_{\mathrm{BLO}_{\rho,\theta}}\\
&=\frac{1}{\sigma}\sum_{j'}\big|Q^{(1)}_{j'}\big|\bigg(1+\frac{r_1}{\rho(x_1)}\bigg)^{\theta}.
\end{split}
\end{equation*}
This is just the desired estimate. Summarizing the estimates derived above($(E)$-2 and $(C)$-1), we can deduce that
\begin{equation*}
\begin{split}
f(x)-\underset{y\in\mathcal{Q}}{\mathrm{ess\,inf}}\,f(y)
&=f(x)-\underset{y\in{Q}^{(1)}_{j'}}{\mathrm{ess\,inf}}\,f(y)+\underset{y\in{Q}^{(1)}_{j'}}{\mathrm{ess\,inf}}\,f(y)
-\underset{y\in\mathcal{Q}}{\mathrm{ess\,inf}}\,f(y)\\
&\leq \sigma\bigg(1+\frac{r_1}{\rho(x_1)}\bigg)^{\theta}+2^d\sigma\bigg(1+\frac{r}{\rho(x_0)}\bigg)^{\theta}\\
&\leq\sigma\bigg(1+\frac{r}{\rho(x_1)}\bigg)^{\theta}+2^d\sigma\bigg(1+\frac{r}{\rho(x_0)}\bigg)^{\theta} ,~~a.e.~x\in {Q}^{(1)}_{j'}\setminus\bigcup_jQ^{(2)}_j.
\end{split}
\end{equation*}
This estimate, together with \eqref{wangh3}, implies that for almost every $x\in {Q}^{(1)}_{j'}\setminus\bigcup_jQ^{(2)}_j$,
\begin{equation*}
\begin{split}
f(x)-\underset{y\in\mathcal{Q}}{\mathrm{ess\,inf}}\,f(y)
&\leq C_0^{\theta}\sigma\bigg(1+\frac{r}{\rho(x_0)}\bigg)^{(N_0+1)\theta}+2^d\sigma\bigg(1+\frac{r}{\rho(x_0)}\bigg)^{\theta}\\
&\leq\big(C_0^{\theta}+2^d\big)\sigma\bigg(1+\frac{r}{\rho(x_0)}\bigg)^{(N_0+1)\theta},
\end{split}
\end{equation*}
which, combined with $(E)$-1, yields that for almost every $x\in {\mathcal{Q}}\setminus\bigcup_jQ^{(2)}_j$,
\begin{equation*}
\begin{split}
f(x)-\underset{y\in\mathcal{Q}}{\mathrm{ess\,inf}}\,f(y)
&\leq\big(C_0^{\theta}+2^d\big)\sigma\bigg(1+\frac{r}{\rho(x_0)}\bigg)^{(N_0+1)\theta}\\
&\leq C_0^{\theta}2\sigma\cdot 2^{d}\bigg(1+\frac{r}{\rho(x_0)}\bigg)^{(N_0+1)\theta}.
\end{split}
\end{equation*}
Moreover, from $(D)$-1 and $(D)$-2, we conclude that
\begin{equation*}
\sum_{j}\big|Q^{(2)}_j\big|\leq\frac{1}{\sigma}\sum_{j'}\big|Q^{(1)}_{j'}\big|\leq\frac{|\mathcal{Q}|}{\sigma^2}.
\end{equation*}
\textbf{Step 3}. We repeat this process indefinitely to obtain a collection of cubes $\{Q^{(k)}_j\}_j$ satisfying the following properties:
\begin{equation*}
\begin{split}
&(A)\mbox{-k}. ~~\mbox{The interior of every cube}~ Q^{(k)}_j ~\mbox{is contained in a unique cube}~ {Q}^{(k-1)}_{j'};\\
&(B)\mbox{-k}. ~~\sigma\bigg(1+\frac{r_{k-1}}{\rho(x_{k-1})}\bigg)^{\theta}
<\frac{1}{|Q^{(k)}_j|}\int_{Q^{(k)}_j}\Big[f(x)-\underset{y\in{Q}^{(k-1)}_{j'}}{\mathrm{ess\,inf}}\,f(y)\Big]\,dx
\leq 2^d\sigma\bigg(1+\frac{r_{k-1}}{\rho(x_{k-1})}\bigg)^{\theta};\\
&(C)\mbox{-k}. ~~0\leq\underset{y\in{Q}^{(k)}_{j}}{\mathrm{ess\,inf}}\,f(y)-\underset{y\in{Q^{(k-1)}_{j'}}}{\mathrm{ess\,inf}}\,f(y)
\leq 2^d\sigma\bigg(1+\frac{r_{k-1}}{\rho(x_{k-1})}\bigg)^{\theta};\\
&(D)\mbox{-k}. ~~\sum_{j}\big|Q^{(k)}_j\big|\leq\frac{1}{\sigma}\sum_{j'}\big|Q^{(k-1)}_{j'}\big|;\\
&(E)\mbox{-k}. ~~f(x)-\underset{y\in{Q}^{(k-1)}_{j'}}{\mathrm{ess\,inf}}\,f(y)
\leq\sigma\bigg(1+\frac{r_{k-1}}{\rho(x_{k-1})}\bigg)^{\theta},~~a.e.~x\in {Q}^{(k-1)}_{j'}\setminus\bigcup_jQ^{(k)}_j.
\end{split}
\end{equation*}
Here ${Q}^{(k-1)}_{j'}$ denotes the cube centered at $x_{k-1}$ with side length $r_{k-1}$. By induction, from the previous proof, it actually follows that
\begin{equation*}
f(x)-\underset{y\in\mathcal{Q}}{\mathrm{ess\,inf}}\,f(y)
\leq C_0^{\theta}k\sigma\cdot 2^d\bigg(1+\frac{r}{\rho(x_0)}\bigg)^{(N_0+1)\theta},~~a.e.~x\in \mathcal{Q}\setminus\bigcup_{\ell}Q^{(k)}_{\ell},
\end{equation*}
and
\begin{equation}\label{www}
\sum_{\ell}\big|Q^{(k)}_{\ell}\big|\leq\frac{|\mathcal{Q}|}{\sigma^{k}},\quad k=1,2,3,\dots.
\end{equation}
Therefore
\begin{equation}\label{hhh}
\Big\{x\in \mathcal{Q}:\Big[f(x)-\underset{y\in\mathcal{Q}}{\mathrm{ess\,inf}}\,f(y)\Big]>C_0^{\theta}k\sigma 2^d\Big(1+\frac{r}{\rho(x_0)}\Big)^{(N_0+1)\theta}\Big\}
\subseteq\bigcup_{\ell}Q^{(k)}_{\ell},k=1,2,3,\dots.
\end{equation}
\textbf{Step 4}. Since
\begin{equation*}
(0,\infty)=\bigcup_{k=0}^{\infty}
\bigg(C_0^{\theta}k\sigma2^d\Big(1+\frac{r}{\rho(x_0)}\Big)^{(N_0+1)\theta},C_0^{\theta}(k+1)\sigma2^d\Big(1+\frac{r}{\rho(x_0)}\Big)^{(N_0+1)\theta}\bigg],
\end{equation*}
then for each fixed $\lambda\in(0,\infty)$, we can write
\begin{equation*}
C_0^{\theta}k\sigma2^d\Big(1+\frac{r}{\rho(x_0)}\Big)^{(N_0+1)\theta}<\lambda\leq C_0^{\theta}(k+1)\sigma2^d\Big(1+\frac{r}{\rho(x_0)}\Big)^{(N_0+1)\theta}
\end{equation*}
for some $k\geq0$, and hence
\begin{equation*}
\begin{split}
&\Big|\Big\{x\in \mathcal{Q}:\Big[f(x)-\underset{y\in\mathcal{Q}}{\mathrm{ess\,inf}}\,f(y)\Big]>\lambda\Big\}\Big|\\
&\leq\Big|\Big\{x\in \mathcal{Q}:\Big[f(x)-\underset{y\in\mathcal{Q}}{\mathrm{ess\,inf}}\,f(y)\Big]>C_0^{\theta}k\sigma 2^d\Big(1+\frac{r}{\rho(x_0)}\Big)^{(N_0+1)\theta}\Big\}\Big|\\
&\leq\sum_{\ell}\big|Q_{\ell}^{(k)}\big|\leq\frac{|\mathcal{Q}|}{\sigma^k}\\
&=\sigma|\mathcal{Q}|\cdot\frac{\exp\{-k\log\sigma\}}{\sigma},
\end{split}
\end{equation*}
where in the last two inequalities we have used \eqref{www} and \eqref{hhh}, respectively. Now choose $\sigma=e>1$, we then have
\begin{equation*}
\begin{split}
&\Big|\Big\{x\in \mathcal{Q}:\Big[f(x)-\underset{y\in\mathcal{Q}}{\mathrm{ess\,inf}}\,f(y)\Big]>\lambda\Big\}\Big|\\
&\leq e|\mathcal{Q}|\exp\big\{-(k+1)\big\}\\
&\leq e|\mathcal{Q}|\exp\Big\{-\Big(1+\frac{r}{\rho(x_0)}\Big)^{-(N_0+1)\theta}\frac{\lambda}{C_0^{\theta}2^de}\Big\}.
\end{split}
\end{equation*}
This concludes the proof of Lemma \ref{expblo} for the special case that $f\in\mathrm{BLO}_{\rho,\theta}$ with $\|f\|_{\mathrm{BLO}_{\rho,\theta}}=1$.

\textbf{Step 5}. We now proceed to the general case. In order to do so, we set
\begin{equation*}
\widetilde{f}(x):=\frac{f(x)}{\|f\|_{\mathrm{BLO}_{\rho,\theta}}}.
\end{equation*}
By the definition of $\|\cdot\|_{\mathrm{BLO}_{\rho,\theta}}$, we have
\begin{equation*}
\|\widetilde{f}\|_{\mathrm{BLO}_{\rho,\theta}}=1\quad \&\quad
\underset{y\in\mathcal{Q}}{\mathrm{ess\,inf}}\,f(y)=\|f\|_{\mathrm{BLO}_{\rho,\theta}}
\cdot\underset{y\in\mathcal{Q}}{\mathrm{ess\,inf}}\,\widetilde{f}(y).
\end{equation*}
Hence,
\begin{equation*}
\begin{split}
&\Big|\Big\{x\in \mathcal{Q}:\Big[f(x)-\underset{y\in\mathcal{Q}}{\mathrm{ess\,inf}}\,f(y)\Big]>\lambda\Big\}\Big|\\
&=\Big|\Big\{x\in \mathcal{Q}:\Big[\widetilde{f}(x)-\underset{y\in\mathcal{Q}}{\mathrm{ess\,inf}}\,\widetilde{f}(y)\Big]>
\frac{\lambda}{\|f\|_{\mathrm{BLO}_{\rho,\theta}}}\Big\}\Big|\\
&\leq \overline{C}_1|\mathcal{Q}|\exp\bigg\{-\bigg(1+\frac{r}{\rho(x_0)}\bigg)^{-(N_0+1)\theta}
\frac{\overline{C}_2\lambda}{\|f\|_{\mathrm{BLO}_{\rho,\theta}}}\bigg\},
\end{split}
\end{equation*}
with precise constants
\begin{equation*}
\overline{C}_1=e\quad \& \quad \overline{C}_2=\frac{1}{C_0^{\theta}2^de}.
\end{equation*}
We are done.
\end{proof}

By using Lemma \ref{expblo}, we have the following result, which describes certain exponential integrability for $\mathrm{BLO}_{\rho,\theta}$ functions.
\begin{lem}\label{BMO3}
If $f\in \mathrm{BLO}_{\rho,\theta}(\mathbb R^d)$ with $0<\theta<\infty$, then there exist positive constants $C>0$ and $\gamma>0$ such that for every cube $\mathcal{Q}=Q(x_0,r)$ in $\mathbb R^d$, we have
\begin{equation}\label{wang2}
\begin{split}
&\bigg(\int_{\mathcal{Q}}\exp\bigg[\bigg(1+\frac{r}{\rho(x_0)}\bigg)^{-\theta^{\ast}}\frac{\gamma}{\|f\|_{\mathrm{BLO}_{\rho,\theta}}}
\Big[f(x)-\underset{y\in\mathcal{Q}}{\mathrm{ess\,inf}}\,f(y)\Big]\bigg]dx\bigg)
\leq C\cdot |\mathcal{Q}|,
\end{split}
\end{equation}
where $\theta^{\ast}=(N_0+1)\theta$ and $N_0$ is the constant appearing in Lemma \ref{N0}.
\end{lem}
\begin{proof}
Recall that the following identity
\begin{equation*}
\bigg(\int_{\mathcal{Q}}\exp\big[|f(x)|\big]\,dx\bigg)
=\int_0^{\infty}e^\lambda \big|\big\{x\in \mathcal{Q}:|f(x)|>\lambda\big\}\big|\,d\lambda
\end{equation*}
holds for any cube $\mathcal{Q}$ in $\mathbb R^d$(see, for instance, \cite[Proposition 1.1.4]{grafakos}). Using this identity and Lemma \ref{expblo}, we obtain
\begin{equation*}
\begin{split}
&\bigg(\int_{\mathcal{Q}}\exp\bigg[\bigg(1+\frac{r}{\rho(x_0)}\bigg)^{-\theta^{\ast}}\frac{\gamma}{\|f\|_{\mathrm{BLO}_{\rho,\theta}}}
\Big[f(x)-\underset{y\in\mathcal{Q}}{\mathrm{ess\,inf}}\,f(y)\Big]\bigg]dx\bigg)\\
&=\int_0^{\infty}\exp(\lambda)\cdot
\Big|\Big\{x\in \mathcal{Q}:\Big[f(x)-\underset{y\in\mathcal{Q}}{\mathrm{ess\,inf}}\,f(y)\Big]>\lambda^{\ast}\Big\}\Big|\,d\lambda\\
&\leq \overline{C}_1|\mathcal{Q}|\int_0^{\infty}
\exp(\lambda)\cdot\exp\bigg[-\bigg(1+\frac{r}{\rho(x_0)}\bigg)^{-\theta^{\ast}}\frac{\overline{C}_2 \lambda^{\ast}}{\|f\|_{\mathrm{BLO}_{\rho,\theta}}}\bigg]d\lambda\\
&=\overline{C}_1\cdot|\mathcal{Q}|\int_0^{\infty}\exp(\lambda)\cdot\exp\Big[-\frac{\overline{C}_2\lambda}{\gamma}\Big]d\lambda,
\end{split}
\end{equation*}
where the number $\lambda^{\ast}$ is given by
\begin{equation*}
\lambda^{\ast}:=\frac{\lambda\|f\|_{\mathrm{BLO}_{\rho,\theta}}}{\gamma}\bigg(1+\frac{r}{\rho(x_0)}\bigg)^{\theta^{\ast}}.
\end{equation*}
If we take $\gamma$ small enough so that $0<\gamma<\overline{C}_2$, then the conclusion follows immediately.
\end{proof}

Moreover, we establish some relevant properties for the spaces $\mathcal{C}^{\beta,\ast}_{\rho,\theta}(\mathbb R^d)$ and $\mathrm{BLO}_{\rho,\theta}(\mathbb R^d)$, which extend some known results in the classical BMO and Campanato spaces.
\begin{prop}
Suppose that $f\in \mathcal{C}^{\beta,\ast}_{\rho,\theta}(\mathbb R^d)$ with $0<\theta<\infty$ and $0<\beta<1$. Then for any $\gamma>\beta+\theta$, there is a constant $C>0$ depending only on $d,\beta$ and $\gamma$ such that for any ball $\mathcal{B}=B(x_0,r)$ in $\mathbb R^d$, we have
\begin{equation*}
\int_{\mathbb R^d}\frac{\Big[f(x)-\underset{y\in\mathcal{B}}{\mathrm{ess\,inf}}\,f(y)\Big]}{r^{d+\gamma}+|x-x_0|^{d+\gamma}}\,dx
\leq C\cdot\frac{\|f\|_{\mathcal{C}^{\beta,\ast}_{\rho,\theta}}}{r^{\gamma-\beta}}\bigg(1+\frac{r}{\rho(x_0)}\bigg)^{\theta}.
\end{equation*}
\end{prop}
\begin{proof}
Suppose that $f\in \mathcal{C}^{\beta,\ast}_{\rho,\theta}(\mathbb R^d)$ with $0<\theta<\infty$ and $0<\beta<1$. We then decompose $\mathbb R^d$ into a geometrically increasing sequence of concentric balls, and obtain
\begin{equation*}
\begin{split}
\int_{\mathbb R^d}\frac{\Big[f(x)-\underset{y\in\mathcal{B}}{\mathrm{ess\,inf}}\,f(y)\Big]}{r^{d+\gamma}+|x-x_0|^{d+\gamma}}\,dx
&=\int_{B(x_0,r)}\frac{\Big[f(x)-\underset{y\in\mathcal{B}}{\mathrm{ess\,inf}}\,f(y)\Big]}{r^{d+\gamma}+|x-x_0|^{d+\gamma}}\,dx\\
&+\sum_{j=1}^{\infty}\int_{B(x_0,2^jr)\setminus B(x_0,2^{j-1}r)}
\frac{\Big[f(x)-\underset{y\in\mathcal{B}}{\mathrm{ess\,inf}}\,f(y)\Big]}{r^{d+\gamma}+|x-x_0|^{d+\gamma}}\,dx\\
&:=\mathrm{I}+\mathrm{II}.
\end{split}
\end{equation*}
For the first term, we have
\begin{equation*}
\begin{split}
\mathrm{I}&\leq\frac{|B(0,1)|}{r^{\gamma}}
\frac{1}{|B(x_0,r)|}\int_{B(x_0,r)}\Big[f(x)-\underset{y\in\mathcal{B}}{\mathrm{ess\,inf}}\,f(y)\Big]\,dx\\
&\leq \frac{|B(0,1)|^{1+\beta/d}}{r^{\gamma-\beta}}\|f\|_{\mathcal{C}^{\beta,\ast}_{\rho,\theta}}\bigg(1+\frac{r}{\rho(x_0)}\bigg)^{\theta}.
\end{split}
\end{equation*}
For the second term, we know that $|x-x_0|\geq 2^{j-1}r$ when $x\in B(x_0,2^jr)\setminus B(x_0,2^{j-1}r)$. Consequently,
\begin{equation*}
\begin{split}
\mathrm{II}
&\leq\sum_{j=1}^{\infty}\frac{1}{(2^{j-1}r)^{d+\gamma}}
\int_{B(x_0,2^jr)}\Big[f(x)-\underset{y\in\mathcal{B}}{\mathrm{ess\,inf}}\,f(y)\Big]\,dx.
\end{split}
\end{equation*}
Observe that for an arbitrary fixed ball $\mathcal{B}$,
\begin{equation*}
\begin{split}
&f(x)-\underset{y\in\mathcal{B}}{\mathrm{ess\,inf}}\,f(y)\\
&=f(x)-\underset{y\in 2^j\mathcal{B}}{\mathrm{ess\,inf}}\,f(y)
+\underset{y\in 2^j\mathcal{B}}{\mathrm{ess\,inf}}\,f(y)-\underset{y\in\mathcal{B}}{\mathrm{ess\,inf}}\,f(y)\\
&\leq f(x)-\underset{y\in 2^j\mathcal{B}}{\mathrm{ess\,inf}}\,f(y),\quad j=1,2,3,\dots.
\end{split}
\end{equation*}
Hence
\begin{equation*}
\begin{split}
\mathrm{II}&\leq\sum_{j=1}^{\infty}\frac{1}{(2^{j-1}r)^{d+\gamma}}
\int_{B(x_0,2^jr)}\Big[f(x)-\underset{y\in 2^j\mathcal{B}}{\mathrm{ess\,inf}}\,f(y)\Big]\,dx\\
&\leq\sum_{j=1}^{\infty}\frac{|B(x_0,2^jr)|^{1+\beta/d}}{(2^{j-1}r)^{d+\gamma}}
\|f\|_{\mathcal{C}^{\beta,\ast}_{\rho,\theta}}\bigg(1+\frac{2^jr}{\rho(x_0)}\bigg)^{\theta}\\
&\leq\sum_{j=1}^{\infty}(2^j)^{\beta-\gamma}\cdot(2^j)^{\theta}\frac{2^{d+\gamma}|B(0,1)|^{1+\beta/d}}{r^{\gamma-\beta}}
\|f\|_{\mathcal{C}^{\beta,\ast}_{\rho,\theta}}\bigg(1+\frac{r}{\rho(x_0)}\bigg)^{\theta}.
\end{split}
\end{equation*}
Note that $\beta-\gamma+\theta<0$, so the desired result follows immediately.
\end{proof}

We can also obtain analogous estimates for the space $\mathrm{BLO}_{\rho,\theta}(\mathbb R^d)$.
\begin{prop}
Suppose that $f\in\mathrm{BLO}_{\rho,\theta}(\mathbb R^d)$ with $0<\theta<\infty$. Then for any $\gamma>\theta$, there is a constant $C>0$ depending only on $d$ and $\gamma$ such that for any cube $\mathcal{Q}=Q(x_0,r)$ in $\mathbb R^d$, we have
\begin{equation*}
\int_{\mathbb R^d}\frac{\Big[f(x)-\underset{y\in\mathcal{B}}{\mathrm{ess\,inf}}\;f(y)\Big]}{r^{d+\gamma}+|x-x_0|^{d+\gamma}}\,dx
\leq C\cdot\frac{\|f\|_{\mathrm{BLO}_{\rho,\theta}}}{r^{\gamma}}\bigg(1+\frac{r}{\rho(x_0)}\bigg)^{\theta}.
\end{equation*}
\end{prop}
We omit the proof here.

\section{Main theorems}

Let $N_0$ be the same constant as in Lemma \ref{N0} and let $\eta$ be the same number as in Lemma \ref{comparelem}. We are now in a position to give the main results of this paper.
\begin{thm}\label{mainthm1}
Let $1\leq p<\infty$ and $\omega\in A^{\rho,\theta_2}_p(\mathbb R^d)$ with $0<\theta_2<\infty$. Then the following
statements are true.
\begin{enumerate}
\item If $f\in \mathrm{BLO}_{\rho,\theta_1}(\mathbb R^d)$ with $0<\theta_1<\infty$, then for any cube $\mathcal{Q}=Q(x_0,r)\subset\mathbb R^d$,
\begin{equation*}
\begin{split}
&\bigg(\frac{1}{\omega(\mathcal{Q})}\int_{\mathcal{Q}}\Big[f(x)-\underset{y\in\mathcal{Q}}{\mathrm{ess\,inf}}\,f(y)\Big]^p\omega(x)\,dx\bigg)^{1/p}\\
&\leq C[\omega]_{A^{\rho,\theta_2}_p}\bigg(1+\frac{r}{\rho(x_0)}\bigg)^{(N_0+1)\theta_1+\eta/p}\|f\|_{\mathrm{BLO}_{\rho,\theta_1}}.
\end{split}
\end{equation*}
\item Conversely, if there exists a constant $C>0$ such that for any cube $\mathcal{Q}=Q(x_0,r)\subset\mathbb R^d$,
\begin{equation}\label{assum1}
\bigg(\frac{1}{\omega(\mathcal{Q})}\int_{\mathcal{Q}}\Big[f(x)-\underset{y\in\mathcal{Q}}{\mathrm{ess\,inf}}\,f(y)\Big]^p\omega(x)\,dx\bigg)^{1/p}\leq C[\omega]_{A^{\rho,\theta_2}_p}\bigg(1+\frac{r}{\rho(x_0)}\bigg)^{\theta_1-\theta_2}
\end{equation}
holds for some $\theta_1>0$, then $f\in\mathrm{BLO}_{\rho,\theta_1}(\mathbb R^d)$, and
\begin{equation*}
\|f\|_{\mathrm{BLO}_{\rho,\theta_1}}\leq C[\omega]_{A^{\rho,\theta_2}_p}.
\end{equation*}
\end{enumerate}
\end{thm}

\begin{thm}\label{mainthm2}
Let $1\leq p<q<\infty$ and $\omega\in A^{\rho,\theta_2}_{p,q}(\mathbb R^d)$ with $0<\theta_2<\infty$. Then the following
statements are true.
\begin{enumerate}
\item If $f\in \mathrm{BLO}_{\rho,\theta_1}(\mathbb R^d)$ with $0<\theta_1<\infty$, then for any cube $\mathcal{Q}=Q(x_0,r)\subset\mathbb R^d$,
\begin{equation*}
\begin{split}
&\bigg(\frac{1}{\omega^q(\mathcal{Q})}\int_{\mathcal{Q}}
\Big[f(x)-\underset{y\in\mathcal{Q}}{\mathrm{ess\,inf}}\,f(y)\Big]^q\omega(x)^q\,dx\bigg)^{1/q}\\
&\leq C[\omega]_{A^{\rho,\theta_2}_{p,q}}\bigg(1+\frac{r}{\rho(x_0)}\bigg)^{(N_0+1)\theta_1+\eta/q}\|f\|_{\mathrm{BLO}_{\rho,\theta_1}}.
\end{split}
\end{equation*}
\item Conversely, if there exists a constant $C>0$ such that for any cube $\mathcal{Q}=Q(x_0,r)\subset\mathbb R^d$,
\begin{equation}\label{assum2}
\bigg(\frac{1}{\omega^q(\mathcal{Q})}\int_{\mathcal{Q}}\Big[f(x)-\underset{y\in\mathcal{Q}}{\mathrm{ess\,inf}}\,f(y)\Big]^q\omega(x)^q\,dx\bigg)^{1/q}\leq C[\omega]_{A^{\rho,\theta_2}_{p,q}}\bigg(1+\frac{r}{\rho(x_0)}\bigg)^{\theta_1-\theta_2}
\end{equation}
holds for some $\theta_1>0$, then $f\in \mathrm{BLO}_{\rho,\theta_1}(\mathbb R^d)$, and
\begin{equation*}
\|f\|_{\mathrm{BLO}_{\rho,\theta_1}}\leq C[\omega]_{A^{\rho,\theta_2}_{p,q}}.
\end{equation*}
\end{enumerate}
\end{thm}
\begin{proof}[Proof of Theorem $\ref{mainthm1}$]
(1) Let $f\in \mathrm{BLO}_{\rho,\theta_1}(\mathbb R^d)$ with $0<\theta_1<\infty$. According to Lemma \ref{expblo}, there are two constants $\overline{C}_1,\overline{C}_2>0$ such that for any $\lambda>0$ and for any cube $\mathcal{Q}\subset\mathbb R^d$,
\begin{equation*}
\begin{split}
&\Big|\Big\{x\in \mathcal{Q}:\Big[f(x)-\underset{y\in\mathcal{Q}}{\mathrm{ess\,inf}}\,f(y)\Big]>\lambda\Big\}\Big|\\
&\leq \overline{C}_1|\mathcal{Q}|\exp\bigg\{-\bigg(1+\frac{r}{\rho(x_0)}\bigg)^{-(N_0+1)\theta_1}\frac{\overline{C}_2\lambda}{\|f\|_{\mathrm{BLO}_{\rho,\theta_1}}}\bigg\}.
\end{split}
\end{equation*}
Since $\omega\in A^{\rho,\theta_2}_p(\mathbb R^d)$ with $0<\theta_2<\infty$ and $1\leq p<\infty$, by using Lemma \ref{comparelem}, we get
\begin{equation*}
\begin{split}
&\omega\Big(\Big\{x\in \mathcal{Q}:\Big[f(x)-\underset{y\in\mathcal{Q}}{\mathrm{ess\,inf}}\,f(y)\Big]>\lambda\Big\}\Big)\\
&\leq C\cdot \overline{C}_1^{\delta}\omega(\mathcal{Q})
\exp\bigg\{-\bigg(1+\frac{r}{\rho(x_0)}\bigg)^{-(N_0+1)\theta_1}\frac{\overline{C}_2\delta\lambda}{\|f\|_{\mathrm{BLO}_{\rho,\theta_1}}}\bigg\}
\times\bigg(1+\frac{r}{\rho(x_0)}\bigg)^{\eta}.
\end{split}
\end{equation*}
Hence, for any cube $\mathcal{Q}\subset\mathbb R^d$,
\begin{align*}
&\bigg(\frac{1}{\omega(\mathcal{Q})}\int_{\mathcal{Q}}\Big[f(x)-\underset{y\in\mathcal{Q}}{\mathrm{ess\,inf}}\,f(y)\Big]^p\omega(x)\,dx\bigg)^{1/p}\nonumber\\
&=\bigg(\frac{1}{\omega(\mathcal{Q})}\int_0^{\infty}p\lambda^{p-1}\omega\Big(\Big\{x\in \mathcal{Q}:\Big[f(x)-\underset{y\in\mathcal{Q}}{\mathrm{ess\,inf}}\,f(y)\Big]>\lambda\Big\}\Big)\,d\lambda\bigg)^{1/p}\nonumber\\
&\leq\bigg(C\cdot \overline{C}_1^{\delta}\int_0^{\infty}p\lambda^{p-1}\exp\bigg\{-\bigg(1+\frac{r}{\rho(x_0)}\bigg)^{-(N_0+1)\theta_1}
\frac{\overline{C}_2\delta\lambda}{\|f\|_{\mathrm{BLO}_{\rho,\theta_1}}}\bigg\}d\lambda\bigg)^{1/p}\nonumber\\
\end{align*}
\begin{align*}
&\times\bigg(1+\frac{r}{\rho(x_0)}\bigg)^{\eta/p}.\nonumber\\
\end{align*}
By making the substitution
\begin{equation*}
\mu=\bigg(1+\frac{r}{\rho(x_0)}\bigg)^{-(N_0+1)\theta_1}\frac{\overline{C}_2\delta\lambda}{\|f\|_{\mathrm{BLO}_{\rho,\theta_1}}},
\end{equation*}
we can deduce that
\begin{align}\label{11}
&\bigg(\frac{1}{\omega(\mathcal{Q})}\int_{\mathcal{Q}}\Big[f(x)-\underset{y\in\mathcal{Q}}{\mathrm{ess\,inf}}\,f(y)\Big]^p
\omega(x)\,dx\bigg)^{1/p}\nonumber\\
&\leq\Big(C\cdot \overline{C}_1^{\delta}p\Big)^{1/p}
\bigg(1+\frac{r}{\rho(x_0)}\bigg)^{(N_0+1)\theta_1}\frac{\|f\|_{\mathrm{BLO}_{\rho,\theta_1}}}{\overline{C}_2\delta}
\times\bigg(1+\frac{r}{\rho(x_0)}\bigg)^{\eta/p}\nonumber\\
&\times\bigg(\int_0^{\infty}\mu^{p-1}e^{-\mu}\,d\mu\bigg)^{1/p}\nonumber\\
&\leq\big(C\cdot p\Gamma(p)\big)^{1/p}
\cdot\frac{\overline{C}_1^{\delta/p}}{\overline{C}_2\delta}\bigg(1+\frac{r}{\rho(x_0)}\bigg)^{(N_0+1)\theta_1+\eta/p}\|f\|_{\mathrm{BLO}_{\rho,\theta_1}}.
\end{align}
This gives the desired inequality. Let us now turn to the proof of $(2)$. The proof of $(2)$ will be divided into two cases.

\textbf{Case 1.} When $1<p<\infty$, by using H\"older's inequality, the condition $\omega\in A^{\rho,\theta_2}_p(\mathbb R^d)$ and \eqref{assum1}, we obtain
\begin{align}\label{12}
&\frac{1}{|\mathcal{Q}|}\int_{\mathcal{Q}}\Big[f(x)-\underset{y\in\mathcal{Q}}{\mathrm{ess\,inf}}\,f(y)\Big]\,dx\nonumber\\
&=\frac{1}{|\mathcal{Q}|}\int_{\mathcal{Q}}\Big[f(x)-\underset{y\in\mathcal{Q}}{\mathrm{ess\,inf}}\,f(y)\Big]\omega(x)^{1/p}
\cdot\omega(x)^{-1/p}\,dx\nonumber\\
&\leq\frac{1}{|\mathcal{Q}|}\bigg(\int_{\mathcal{Q}}\Big[f(x)-\underset{y\in\mathcal{Q}}{\mathrm{ess\,inf}}\,f(y)\Big]^p\omega(x)\,dx\bigg)^{1/p}
\bigg(\int_{\mathcal{Q}}\omega(x)^{-{p'}/p}\,dx\bigg)^{1/{p'}}\nonumber\\
&\leq C\bigg(1+\frac{r}{\rho(x_0)}\bigg)^{\theta_1-\theta_2}\nonumber\\
&\times\bigg(\frac{1}{|\mathcal{Q}|}\int_{\mathcal{Q}}\omega(x)\,dx\bigg)^{1/p}\bigg(\frac{1}{|\mathcal{Q}|}\int_{\mathcal{Q}}\omega(x)^{-{p'}/p}\,dx\bigg)^{1/{p'}}\nonumber\\
&\leq C[\omega]_{A^{\rho,\theta_2}_p}\bigg(1+\frac{r}{\rho(x_0)}\bigg)^{\theta_1}.
\end{align}
\textbf{Case 2.} When $p=1$, then it follows directly from the condition $\omega\in A^{\rho,\theta_2}_1(\mathbb R^d)$ and \eqref{assum1} that
\begin{align}\label{13}
&\frac{1}{|\mathcal{Q}|}\int_{\mathcal{Q}}\Big[f(x)-\underset{y\in\mathcal{Q}}{\mathrm{ess\,inf}}\,f(y)\Big]\,dx\nonumber\\
&=\frac{1}{|\mathcal{Q}|}\int_{\mathcal{Q}}\Big[f(x)-\underset{y\in\mathcal{Q}}{\mathrm{ess\,inf}}\,f(y)\Big]\omega(x)
\cdot\omega(x)^{-1}\,dx\nonumber\\
&\leq\frac{1}{|\mathcal{Q}|}\bigg(\int_{\mathcal{Q}}\Big[f(x)-\underset{y\in\mathcal{Q}}{\mathrm{ess\,inf}}\,f(y)\Big]\omega(x)\,dx\bigg)
\bigg(\underset{x\in \mathcal{Q}}{\mbox{ess\,sup}}\,\omega(x)^{-1}\bigg)\nonumber\\
&\leq C\bigg(1+\frac{r}{\rho(x_0)}\bigg)^{\theta_1-\theta_2}
\times\bigg(\frac{1}{|\mathcal{Q}|}\int_{\mathcal{Q}}\omega(x)\,dx\bigg)\bigg(\underset{x\in \mathcal{Q}}{\mbox{ess\,inf}}\,\omega(x)\bigg)^{-1}\nonumber\\
&\leq C[\omega]_{A^{\rho,\theta_2}_1}\bigg(1+\frac{r}{\rho(x_0)}\bigg)^{\theta_1}.
\end{align}
Collecting the above estimates \eqref{12} and \eqref{13}, we conclude the proof of Theorem \ref{mainthm1}.
\end{proof}

\begin{proof}[Proof of Theorem \ref{mainthm2}]
(1) Let $f\in \mathrm{BLO}_{\rho,\theta_1}(\mathbb R^d)$ with $0<\theta_1<\infty$. According to Lemma \ref{expblo}, there are two constants $\overline{C}_1,\overline{C}_2>0$ such that for any $\lambda>0$ and for any cube $\mathcal{Q}\subset\mathbb R^d$,
\begin{equation*}
\begin{split}
&\Big|\Big\{x\in \mathcal{Q}:\Big[f(x)-\underset{y\in\mathcal{Q}}{\mathrm{ess\,inf}}\,f(y)\Big]>\lambda\Big\}\Big|\\
&\leq \overline{C}_1|\mathcal{Q}|
\exp\bigg\{-\bigg(1+\frac{r}{\rho(x_0)}\bigg)^{-(N_0+1)\theta_1}\frac{\overline{C}_2\lambda}{\|f\|_{\mathrm{BLO}_{\rho,\theta_1}}}\bigg\}.
\end{split}
\end{equation*}
Since $\omega\in A^{\rho,\theta_2}_{p,q}(\mathbb R^d)$ with $0<\theta_2<\infty$ and $1\leq p<q<\infty$, by using Lemma \ref{Apq} and Lemma \ref{comparelem}, we have
\begin{equation*}
\begin{split}
&\omega^q\Big(\Big\{x\in \mathcal{Q}:\Big[f(x)-\underset{y\in\mathcal{Q}}{\mathrm{ess\,inf}}\,f(y)\Big]>\lambda\Big\}\Big)\\
&\leq C\cdot \overline{C}_1^{\delta}\omega^q(\mathcal{Q})
\exp\bigg\{-\bigg(1+\frac{r}{\rho(x_0)}\bigg)^{-(N_0+1)\theta_1}\frac{\overline{C}_2\delta\lambda}{\|f\|_{\mathrm{BLO}_{\rho,\theta_1}}}\bigg\}
\times\bigg(1+\frac{r}{\rho(x_0)}\bigg)^{\eta}.
\end{split}
\end{equation*}
Here the symbol $\omega^{\gamma}(E)$ for $\gamma>0$ is given in Section \ref{intro}. Hence, for any cube $\mathcal{Q}\subset\mathbb R^d$,
\begin{align*}
&\bigg(\frac{1}{\omega^q(\mathcal{Q})}\int_{\mathcal{Q}}
\Big[f(x)-\underset{y\in\mathcal{Q}}{\mathrm{ess\,inf}}\,f(y)\Big]^q\omega(x)^q\,dx\bigg)^{1/q}\nonumber\\
&=\bigg(\frac{1}{\omega^q(\mathcal{Q})}\int_0^{\infty}q\lambda^{q-1}\omega^q\Big(\Big\{x\in \mathcal{Q}:\Big[f(x)-\underset{y\in\mathcal{Q}}{\mathrm{ess\,inf}}\,f(y)\Big]>\lambda\Big\}\Big)\,d\lambda\bigg)^{1/q}\nonumber\\
&\leq\bigg(C\cdot \overline{C}_1^{\delta}\int_0^{\infty}q\lambda^{q-1}\exp\bigg\{-\bigg(1+\frac{r}{\rho(x_0)}\bigg)^{-(N_0+1)\theta_1}
\frac{\overline{C}_2\delta\lambda}{\|f\|_{\mathrm{BLO}_{\rho,\theta_1}}}\bigg\}d\lambda\bigg)^{1/q}\nonumber\\
\end{align*}
\begin{align*}
&\times\bigg(1+\frac{r}{\rho(x_0)}\bigg)^{\eta/q}.\nonumber\\
\end{align*}

By making the substitution
\begin{equation*}
\nu=\bigg(1+\frac{r}{\rho(x_0)}\bigg)^{-(N_0+1)\theta_1}
\frac{\overline{C}_2\delta\lambda}{\|f\|_{\mathrm{BLO}_{\rho,\theta_1}}},
\end{equation*}
we can see that
\begin{align}\label{21}
&\bigg(\frac{1}{\omega^q(\mathcal{Q})}\int_{\mathcal{Q}}
\Big[f(x)-\underset{y\in\mathcal{Q}}{\mathrm{ess\,inf}}\,f(y)\Big]^q\omega(x)^q\,dx\bigg)^{1/q}\nonumber\\
&\leq \big(C\cdot q\big)^{1/q}
\cdot\frac{\overline{C}_1^{\delta/q}}{\overline{C}_2\delta}\bigg(1+\frac{r}{\rho(x_0)}\bigg)^{(N_0+1)\theta_1+\eta/q}\|f\|_{\mathrm{BLO}_{\rho,\theta_1}}
\times\bigg(\int_0^{\infty}\nu^{q-1}e^{-\nu}\,d\nu\bigg)^{1/q}\nonumber\\
&\leq \big(C\cdot q\Gamma(q)\big)^{1/q}
\cdot\frac{\overline{C}_1^{\delta/q}}{\overline{C}_2\delta}
\bigg(1+\frac{r}{\rho(x_0)}\bigg)^{(N_0+1)\theta_1+\eta/q}\|f\|_{\mathrm{BLO}_{\rho,\theta_1}}.
\end{align}
This yields the desired estimate. Let us now turn to the proof of $(2)$. As before, the proof of $(2)$ will be divided into two cases.

\textbf{Case 1.} When $1<p<\infty$, it then follows directly from the H\"{o}lder inequality that
\begin{align}\label{22}
&\frac{1}{|\mathcal{Q}|}\int_{\mathcal{Q}}\Big[f(x)-\underset{y\in\mathcal{Q}}{\mathrm{ess\,inf}}\,f(y)\Big]\,dx\\
&=\frac{1}{|\mathcal{Q}|}\int_{\mathcal{Q}}\Big[f(x)-\underset{y\in\mathcal{Q}}{\mathrm{ess\,inf}}\,f(y)\Big]\omega(x)\cdot\omega(x)^{-1}\,dx\nonumber\\
&\leq\frac{1}{|\mathcal{Q}|}\bigg(\int_{\mathcal{Q}}\Big[f(x)-\underset{y\in\mathcal{Q}}{\mathrm{ess\,inf}}\,f(y)\Big]^p\omega(x)^p\,dx\bigg)^{1/p}
\bigg(\int_{\mathcal{Q}}\omega(x)^{-{p'}}\,dx\bigg)^{1/{p'}}\nonumber.
\end{align}
Moreover, by using the H\"{o}lder inequality again, we can see that when $1\leq p<q$,
\begin{equation}\label{above}
\bigg(\frac{1}{|\mathcal{Q}|}\int_{\mathcal{Q}}\Big[f(x)-\underset{y\in\mathcal{Q}}{\mathrm{ess\,inf}}\,f(y)\Big]^p\omega(x)^p\,dx\bigg)^{1/p}
\leq\bigg(\frac{1}{|\mathcal{Q}|}\int_{\mathcal{Q}}\Big[f(x)-\underset{y\in\mathcal{Q}}{\mathrm{ess\,inf}}\,f(y)\Big]^q\omega(x)^q\,dx\bigg)^{1/q}.
\end{equation}
Substituting the above inequality into \eqref{22}, we thus obtain
\begin{align}\label{221}
&\frac{1}{|\mathcal{Q}|}\int_{\mathcal{Q}}\Big[f(x)-\underset{y\in\mathcal{Q}}{\mathrm{ess\,inf}}\,f(y)\Big]\,dx\\
&\leq \frac{|\mathcal{Q}|^{1/p-1/q}}{|\mathcal{Q}|}\bigg(\int_{\mathcal{Q}}
\Big[f(x)-\underset{y\in\mathcal{Q}}{\mathrm{ess\,inf}}\,f(y)\Big]^q\omega(x)^q\,dx\bigg)^{1/q}
\bigg(\int_{\mathcal{Q}}\omega(x)^{-{p'}}\,dx\bigg)^{1/{p'}}\nonumber\\
&\leq C\bigg(1+\frac{r}{\rho(x_0)}\bigg)^{\theta_1-\theta_2}\nonumber\\
&\times\bigg(\frac{1}{|\mathcal{Q}|}\int_{\mathcal{Q}}\omega(x)^q\,dx\bigg)^{1/q}
\bigg(\frac{1}{|\mathcal{Q}|}\int_{\mathcal{Q}}\omega(x)^{-{p'}}\,dx\bigg)^{1/{p'}}\nonumber\\
&\leq C[\omega]_{A^{\rho,\theta_2}_{p,q}}\bigg(1+\frac{r}{\rho(x_0)}\bigg)^{\theta_1},
\end{align}
where in the last two inequalities we have used \eqref{assum2} and the definition of $A^{\rho,\theta_2}_{p,q}$, respectively.

\textbf{Case 2.} When $p=1$ and $1<q<\infty$, then we have
\begin{align*}
&\frac{1}{|\mathcal{Q}|}\int_{\mathcal{Q}}\Big[f(x)-\underset{y\in\mathcal{Q}}{\mathrm{ess\,inf}}\,f(y)\Big]\,dx\\
&=\frac{1}{|\mathcal{Q}|}\int_{\mathcal{Q}}\Big[f(x)-\underset{y\in\mathcal{Q}}{\mathrm{ess\,inf}}\,f(y)\Big]
\omega(x)\cdot\omega(x)^{-1}\,dx\nonumber\\
&\leq\frac{1}{|\mathcal{Q}|}\bigg(\int_{\mathcal{Q}}\Big[f(x)-\underset{y\in\mathcal{Q}}{\mathrm{ess\,inf}}\,f(y)\Big]\omega(x)\,dx\bigg)
\bigg(\underset{x\in \mathcal{Q}}{\mbox{ess\,sup}}\,\omega(x)^{-1}\bigg).\nonumber\\
\end{align*}
From the previous estimate \eqref{above}(with $p=1$), it actually follows that
\begin{align}\label{222}
&\frac{1}{|\mathcal{Q}|}\int_{\mathcal{Q}}\Big[f(x)-\underset{y\in\mathcal{Q}}{\mathrm{ess\,inf}}\,f(y)\Big]\,dx\nonumber\\
&\leq \frac{|\mathcal{Q}|^{1-1/q}}{|\mathcal{Q}|}
\bigg(\int_{\mathcal{Q}}\Big[f(x)-\underset{y\in\mathcal{Q}}{\mathrm{ess\,inf}}\,f(y)\Big]^q\omega(x)^q\,dx\bigg)^{1/q}
\bigg(\underset{x\in \mathcal{Q}}{\mbox{ess\,sup}}\,\omega(x)^{-1}\bigg)\nonumber\\
&\leq C\bigg(1+\frac{r}{\rho(x_0)}\bigg)^{\theta_1-\theta_2}\nonumber
\times\bigg(\frac{1}{|\mathcal{Q}|}\int_{\mathcal{Q}}\omega(x)^q\,dx\bigg)^{1/q}
\bigg(\underset{x\in \mathcal{Q}}{\mbox{ess\,inf}}\,\omega(x)\bigg)^{-1}\nonumber\\
&\leq C[\omega]_{A^{\rho,\theta_2}_{1,q}}\bigg(1+\frac{r}{\rho(x_0)}\bigg)^{\theta_1},
\end{align}
where in the last two inequalities we have used \eqref{assum2} and the definition of $A^{\rho,\theta_2}_{1,q}$, respectively. Collecting the above estimates \eqref{221} and \eqref{222}, we finish the proof of Theorem \ref{mainthm2}.
\end{proof}

For any cube $\mathcal{Q}$ (or ball $\mathcal{B}$) in $\mathbb R^d$ and for $1<p<\infty$, by using H\"{o}lder's inequality, we have
\begin{equation*}
|\mathcal{Q}|=\int_{\mathcal{Q}}\omega(x)\cdot\omega(x)^{-1}\,dx
\leq\bigg(\int_{\mathcal{Q}}\omega(x)^p\,dx\bigg)^{1/p}\bigg(\int_{\mathcal{Q}}\omega(x)^{-p'}\,dx\bigg)^{1/{p'}}.
\end{equation*}
By the definition of $A^{\rho,\theta_2}_{p,q}$ weights, we get
\begin{equation*}
\begin{split}
\bigg(\int_{\mathcal{Q}}\omega(x)^q\,dx\bigg)^{1/q}
\bigg(\int_{\mathcal{Q}}\omega(x)^{-{p'}}\,dx\bigg)^{1/{p'}}
&\leq [\omega]_{A^{\rho,\theta_2}_{p,q}}|\mathcal{Q}|^{1/q+1/{p'}}\bigg(1+\frac{r}{\rho(x_0)}\bigg)^{\theta_2}.
\end{split}
\end{equation*}
Consequently,
\begin{equation}\label{wangh1}
\begin{split}
\bigg(\int_{\mathcal{Q}}\omega(x)^q\,dx\bigg)^{1/q}
&\leq[\omega]_{A^{\rho,\theta_2}_{p,q}}\frac{|\mathcal{Q}|^{1/q+1/{p'}}}{|\mathcal{Q}|}\bigg(1+\frac{r}{\rho(x_0)}\bigg)^{\theta_2}
\bigg(\int_{\mathcal{Q}}\omega(x)^p\,dx\bigg)^{1/p}\\
&=[\omega]_{A^{\rho,\theta_2}_{p,q}}\bigg(1+\frac{r}{\rho(x_0)}\bigg)^{\theta_2}
|\mathcal{Q}|^{1/q-1/p}\bigg(\int_{\mathcal{Q}}\omega(x)^p\,dx\bigg)^{1/p}.
\end{split}
\end{equation}
We remark that the above estimate also holds for the case $p=1$ and $\omega\in A^{\rho,\theta_2}_{1,q}(\mathbb R^d)$. Indeed, it is immediate that by definition
\begin{equation*}
\begin{split}
\bigg(\int_{\mathcal{Q}}\omega(x)^q\,dx\bigg)^{1/q}
&\leq [\omega]_{A^{\rho,\theta_2}_{1,q}}\bigg(1+\frac{r}{\rho(x_0)}\bigg)^{\theta_2}
|\mathcal{Q}|^{1/q}\underset{x\in \mathcal{Q}}{\mbox{ess\,inf}}\,\omega(x)\\
&\leq [\omega]_{A^{\rho,\theta_2}_{1,q}}\bigg(1+\frac{r}{\rho(x_0)}\bigg)^{\theta_2}
|\mathcal{Q}|^{1/q-1}\bigg(\int_{\mathcal{Q}}\omega(x)\,dx\bigg).
\end{split}
\end{equation*}
On the other hand, for $1\leq p<q$, it follows directly from H\"{o}lder's inequality that
\begin{equation*}
\bigg(\frac{1}{|\mathcal{Q}|}\int_{\mathcal{Q}}\omega(x)^p\,dx\bigg)^{1/p}
\leq\bigg(\frac{1}{|\mathcal{Q}|}\int_{\mathcal{Q}}\omega(x)^q\,dx\bigg)^{1/q},
\end{equation*}
which implies that
\begin{equation}\label{wangh2}
\bigg(\int_{\mathcal{Q}}\omega(x)^q\,dx\bigg)^{1/q}\geq|\mathcal{Q}|^{1/q-1/p}\bigg(\int_{\mathcal{Q}}\omega(x)^p\,dx\bigg)^{1/p}.
\end{equation}

As a consequence of \eqref{wangh1} and \eqref{wangh2}, we then obtain the following conclusions.
\begin{cor}\label{cor1}
Let $1\leq p<q<\infty$ and $\omega\in A^{\rho,\theta_2}_{p,q}(\mathbb R^d)$ with $0<\theta_2<\infty$. Then the following
statements are true.
\begin{enumerate}
\item If $f\in \mathrm{BLO}_{\rho,\theta_1}(\mathbb R^d)$ with $0<\theta_1<\infty$, then for any cube $\mathcal{Q}=Q(x_0,r)\subset\mathbb R^d$,
\begin{equation*}
\begin{split}
&\frac{|\mathcal{Q}|^{1/p-1/q}}{[\omega^p(\mathcal{Q})]^{1/p}}\bigg(\int_{\mathcal{Q}}
\Big[f(x)-\underset{y\in\mathcal{Q}}{\mathrm{ess\,inf}}\,f(y)\Big]^q\omega(x)^q\,dx\bigg)^{1/q}\\
&\leq C[\omega]_{A^{\rho,\theta_2}_{p,q}}\bigg(1+\frac{r}{\rho(x_0)}\bigg)^{(N_0+1)\theta_1+\theta_2+\eta/q}\|f\|_{\mathrm{BLO}_{\rho,\theta_1}}.
\end{split}
\end{equation*}
\item Conversely, if there exists a constant $C>0$ such that for any cube $\mathcal{Q}=Q(x_0,r)\subset\mathbb R^d$,
\begin{equation*}
\frac{|\mathcal{Q}|^{1/p-1/q}}{[\omega^p(\mathcal{Q})]^{1/p}}
\bigg(\int_{\mathcal{Q}}\Big[f(x)-\underset{y\in\mathcal{Q}}{\mathrm{ess\,inf}}\,f(y)\Big]^q\omega(x)^q\,dx\bigg)^{1/q}\leq C[\omega]_{A^{\rho,\theta_2}_{p,q}}\bigg(1+\frac{r}{\rho(x_0)}\bigg)^{\theta_1-\theta_2}
\end{equation*}
holds for some $\theta_1>0$, then $f\in \mathrm{BLO}_{\rho,\theta_1}(\mathbb R^d)$, and
\begin{equation*}
\|f\|_{\mathrm{BLO}_{\rho,\theta_1}}\leq C[\omega]_{A^{\rho,\theta_2}_{p,q}}.
\end{equation*}
\end{enumerate}
\end{cor}

\begin{thm}\label{mainthm3}
Let $1\leq p<\infty$ and $\omega\in A^{\rho,\theta_2}_p(\mathbb R^d)$ with $0<\theta_2<\infty$. Then the following
statements are true.
\begin{enumerate}
\item If $f\in \mathcal{C}^{\beta,\ast}_{\rho,\theta_1}(\mathbb R^d)$ with $0<\beta<1$ and $0<\theta_1<\infty$, then for any ball $\mathcal{B}=B(x_0,r)$ in $\mathbb R^d$,
\begin{equation*}
\begin{split}
&\frac{1}{|\mathcal{B}|^{\beta/d}}\bigg(\frac{1}{\omega(\mathcal{B})}
\int_{\mathcal{B}}\Big[f(x)-\underset{y\in\mathcal{B}}{\mathrm{ess\,inf}}\,f(y)\Big]^p\omega(x)\,dx\bigg)^{1/p}\\
&\leq C[\omega]_{A^{\rho,\theta_2}_p}\bigg(1+\frac{r}{\rho(x_0)}\bigg)^{(N_0+1)\theta_1}\|f\|_{\mathcal{C}^{\beta,\ast}_{\rho,\theta_1}}.
\end{split}
\end{equation*}
\item Conversely, if there exists a constant $C>0$ such that for any ball $\mathcal{B}=B(x_0,r)\subset\mathbb R^d$ and $0<\beta<1$,
\begin{equation*}
\frac{1}{|\mathcal{B}|^{\beta/d}}\bigg(\frac{1}{\omega(\mathcal{B})}
\int_{\mathcal{B}}\Big[f(x)-\underset{y\in\mathcal{B}}{\mathrm{ess\,inf}}\,f(y)\Big]^p\omega(x)\,dx\bigg)^{1/p}\leq C[\omega]_{A^{\rho,\theta_2}_p}\bigg(1+\frac{r}{\rho(x_0)}\bigg)^{\theta_1-\theta_2}
\end{equation*}
holds for some $\theta_1>0$, then $f\in\mathcal{C}^{\beta,\ast}_{\rho,\theta_1}(\mathbb R^d)$, and
\begin{equation*}
\|f\|_{\mathcal{C}^{\beta,\ast}_{\rho,\theta_1}}\leq C[\omega]_{A^{\rho,\theta_2}_p}.
\end{equation*}
\end{enumerate}
\end{thm}

\begin{thm}\label{mainthm4}
Let $1\leq p<q<\infty$ and $\omega\in A^{\rho,\theta_2}_{p,q}(\mathbb R^d)$ with $0<\theta_2<\infty$. Then the following
statements are true.
\begin{enumerate}
\item If $f\in \mathcal{C}^{\beta,\ast}_{\rho,\theta_1}(\mathbb R^d)$ with $0<\beta<1$ and $0<\theta_1<\infty$, then for any ball $\mathcal{B}=B(x_0,r)$ in $\mathbb R^d$,
\begin{equation*}
\begin{split}
&\frac{1}{|\mathcal{B}|^{\beta/d}}\bigg(\frac{1}{\omega^q(\mathcal{B})}
\int_{\mathcal{B}}\Big[f(x)-\underset{y\in\mathcal{B}}{\mathrm{ess\,inf}}\,f(y)\Big]^q\omega(x)^q\,dx\bigg)^{1/q}\\
&\leq C[\omega]_{A^{\rho,\theta_2}_{p,q}}\bigg(1+\frac{r}{\rho(x_0)}\bigg)^{(N_0+1)\theta_1}\|f\|_{\mathcal{C}^{\beta,\ast}_{\rho,\theta_1}}.
\end{split}
\end{equation*}
\item Conversely, if there exists a constant $C>0$ such that for any ball $\mathcal{B}=B(x_0,r)\subset\mathbb R^d$ and $0<\beta<1$,
\begin{equation*}
\frac{1}{|\mathcal{B}|^{\beta/d}}\bigg(\frac{1}{\omega^q(\mathcal{B})}
\int_{\mathcal{B}}\Big[f(x)-\underset{y\in\mathcal{B}}{\mathrm{ess\,inf}}\,f(y)\Big]^q\omega(x)^q\,dx\bigg)^{1/q}\leq C[\omega]_{A^{\rho,\theta_2}_{p,q}}\bigg(1+\frac{r}{\rho(x_0)}\bigg)^{\theta_1-\theta_2}
\end{equation*}
holds for some $\theta_1>0$, then $f\in\mathcal{C}^{\beta,\ast}_{\rho,\theta_1}(\mathbb R^d)$, and
\begin{equation*}
\|f\|_{\mathcal{C}^{\beta,\ast}_{\rho,\theta_1}}\leq C[\omega]_{A^{\rho,\theta_2}_{p,q}}.
\end{equation*}
\end{enumerate}
\end{thm}

\begin{proof}[Proof of Theorem \ref{mainthm3}]
Following along the same lines as that of Theorem \ref{mainthm1}, we can also prove the second part (2). We only need to show the first part (1).
For an arbitrary fixed ball $\mathcal{B}=B(x_0,r)$ with $x_0\in\mathbb R^d$ and $r\in(0,\infty)$, first observe that
\begin{equation*}
\Big[f(x)-\underset{y\in\mathcal{B}}{\mathrm{ess\,inf}}\,f(y)\Big]\leq\underset{y\in\mathcal{B}}{\mathrm{ess\,sup}}\,\big|f(x)-f(y)\big|
\end{equation*}
holds for any $x\in \mathcal{B}$. We also mention that $\mathcal{C}^{\beta,\ast}_{\rho,\theta_1}(\mathbb R^d)$ is a subspace of $\mathcal{C}^{\beta}_{\rho,\theta_1}(\mathbb R^d)$, and
\begin{equation}\label{999}
\|f\|_{\mathcal{C}^{\beta}_{\rho,\theta_1}}\leq 2\|f\|_{\mathcal{C}^{\beta,\ast}_{\rho,\theta_1}}.
\end{equation}
Since
\begin{equation*}
\begin{split}
&\frac{1}{|\mathcal{B}|^{1+\beta/d}}\int_{\mathcal{B}}\big|f(x)-f_{\mathcal{B}}\big|\,dx\\
&=\frac{1}{|\mathcal{B}|^{1+\beta/d}}\int_{\mathcal{B}}\Big|f(x)
-\underset{y\in\mathcal{B}}{\mathrm{ess\,inf}}\,f(y)+\underset{y\in\mathcal{B}}{\mathrm{ess\,inf}}\,f(y)-f_{\mathcal{B}}\Big|\,dx\\
&\leq\frac{1}{|\mathcal{B}|^{1+\beta/d}}\int_{\mathcal{B}}\Big[f(x)-\underset{y\in\mathcal{B}}{\mathrm{ess\,inf}}\,f(y)\Big]\,dx
+\frac{1}{|\mathcal{B}|^{\beta/d}}\Big|\underset{y\in\mathcal{B}}{\mathrm{ess\,inf}}\,f(y)-f_{\mathcal{B}}\Big|\\
&\leq\frac{2}{|\mathcal{B}|^{1+\beta/d}}\int_{\mathcal{B}}\Big[f(x)-\underset{y\in\mathcal{B}}{\mathrm{ess\,inf}}\,f(y)\Big]\,dx,
\end{split}
\end{equation*}
multiplying both sides of the above inequality by $(1+r/{\rho(x_0)})^{-\theta}$ and then taking the supremum over all balls $B\subseteq\mathbb R^d$, we get \eqref{999}. In view of Lemma \ref{beta} and \eqref{999}, one can see that for any $x,y\in \mathcal{B}$,
\begin{equation*}
\begin{split}
|f(x)-f(y)|&\leq C\|f\|_{\mathcal{C}^{\beta}_{\rho,\theta_1}}|x-y|^{\beta}
\bigg(1+\frac{|x-y|}{\rho(x)}+\frac{|x-y|}{\rho(y)}\bigg)^{\theta_1}\\
&\leq C|\mathcal{B}|^{\beta/d}\|f\|_{\mathcal{C}^{\beta,\ast}_{\rho,\theta_1}}
\bigg(1+\frac{2r}{\rho(x)}+\frac{2r}{\rho(y)}\bigg)^{\theta_1}\\
&\leq C|\mathcal{B}|^{\beta/d}\|f\|_{\mathcal{C}^{\beta,\ast}_{\rho,\theta_1}}
\bigg(1+\frac{r}{\rho(x)}+\frac{r}{\rho(y)}\bigg)^{\theta_1}.
\end{split}
\end{equation*}
This, together with the estimate \eqref{wangh3}, gives us that
\begin{equation*}
|f(x)-f(y)|\leq C|\mathcal{B}|^{\beta/d}\|f\|_{\mathcal{C}^{\beta,\ast}_{\rho,\theta_1}}
\bigg(1+\frac{r}{\rho(x_0)}\bigg)^{(N_0+1)\theta_1}.
\end{equation*}
Hence, for any $x\in \mathcal{B}$,
\begin{equation*}
\Big[f(x)-\underset{y\in\mathcal{B}}{\mathrm{ess\,inf}}\,f(y)\Big]
\leq C|\mathcal{B}|^{\beta/d}\|f\|_{\mathcal{C}^{\beta,\ast}_{\rho,\theta_1}}
\bigg(1+\frac{r}{\rho(x_0)}\bigg)^{(N_0+1)\theta_1}.
\end{equation*}
Therefore,
\begin{equation*}
\frac{1}{|\mathcal{B}|^{\beta/d}}\bigg(\frac{1}{\omega(\mathcal{B})}
\int_{\mathcal{B}}\Big[f(x)-\underset{y\in\mathcal{B}}{\mathrm{ess\,inf}}\,f(y)\Big]^p\omega(x)\,dx\bigg)^{1/p}
\leq C\bigg(1+\frac{r}{\rho(x_0)}\bigg)^{(N_0+1)\theta_1}\|f\|_{\mathcal{C}^{\beta,\ast}_{\rho,\theta_1}}.
\end{equation*}
This completes the proof of Theorem \ref{mainthm3}.
\end{proof}

\begin{proof}[Proof of Theorem \ref{mainthm4}]
Following along the same lines as that of Theorem \ref{mainthm2}, we can also prove the second part (2). So we only need to show the first part (1). For an arbitrary fixed ball $\mathcal{B}=B(x_0,r)$, arguing as in the proof of Theorem \ref{mainthm3}, we can also obtain analogous estimate below.
\begin{equation*}
\frac{1}{|\mathcal{B}|^{\beta/d}}\bigg(\frac{1}{\omega^q(\mathcal{B})}
\int_{\mathcal{B}}\Big[f(x)-\underset{y\in\mathcal{B}}{\mathrm{ess\,inf}}\,f(y)\Big]^q\omega(x)^q\,dx\bigg)^{1/q}
\leq C\bigg(1+\frac{r}{\rho(x_0)}\bigg)^{(N_0+1)\theta_1}\|f\|_{\mathcal{C}^{\beta,\ast}_{\rho,\theta_1}}.
\end{equation*}
This concludes the proof of Theorem \ref{mainthm4}.
\end{proof}

In view of the estimates \eqref{wangh1} and \eqref{wangh2}, we immediately obtain the following results.
\begin{cor}\label{cor2}
Let $1\leq p<q<\infty$ and $\omega\in A^{\rho,\theta_2}_{p,q}(\mathbb R^d)$ with $0<\theta_2<\infty$. Then the following
statements are true.
\begin{enumerate}
\item If $f\in \mathcal{C}^{\beta,\ast}_{\rho,\theta_1}(\mathbb R^d)$ with $0<\beta<1$ and $0<\theta_1<\infty$, then for any ball $\mathcal{B}=B(x_0,r)\subset\mathbb R^d$,
\begin{equation*}
\begin{split}
&\frac{|\mathcal{B}|^{1/p-1/q-\beta/d}}{[\omega^p(\mathcal{B})]^{1/p}}
\bigg(\int_{\mathcal{B}}\Big[f(x)-\underset{y\in\mathcal{B}}{\mathrm{ess\,inf}}\,f(y)\Big]^q\omega(x)^q\,dx\bigg)^{1/q}\\
&\leq C[\omega]_{A^{\rho,\theta_2}_{p,q}}\bigg(1+\frac{r}{\rho(x_0)}\bigg)^{(N_0+1)\theta_1+\theta_2}
\|f\|_{\mathcal{C}^{\beta,\ast}_{\rho,\theta_1}}.
\end{split}
\end{equation*}
\item Conversely, if there exists a constant $C>0$ such that for any ball $\mathcal{B}=B(x_0,r)\subset\mathbb R^d$ and $0<\beta<1$,
\begin{equation*}
\frac{|\mathcal{B}|^{1/p-1/q-\beta/d}}{[\omega^p(\mathcal{B})]^{1/p}}
\bigg(\int_{\mathcal{B}}\Big[f(x)-\underset{y\in\mathcal{B}}{\mathrm{ess\,inf}}\,f(y)\Big]^q\omega(x)^q\,dx\bigg)^{1/q}\leq C[\omega]_{A^{\rho,\theta_2}_{p,q}}\bigg(1+\frac{r}{\rho(x_0)}\bigg)^{\theta_1-\theta_2}
\end{equation*}
holds for some $\theta_1>0$, then $f\in \mathcal{C}^{\beta,\ast}_{\rho,\theta_1}(\mathbb R^d)$, and
\begin{equation*}
\|f\|_{\mathcal{C}^{\beta,\ast}_{\rho,\theta_1}}\leq C[\omega]_{A^{\rho,\theta_2}_{p,q}}.
\end{equation*}
\end{enumerate}
\end{cor}

Summarizing the estimates derived above, we finally obtain the following conclusions by the definitions of .

\begin{cor}\label{cor3}
Let $1\leq p<\infty$ and $\omega\in A^{\rho,\infty}_p(\mathbb R^d)$. Then the following
statements are true.
\begin{enumerate}
\item $f\in \mathrm{BLO}_{\rho,\infty}(\mathbb R^d)$ if and only if there exists a constant $C>0$ such that, for any cube $\mathcal{Q}=Q(x_0,r)\subset\mathbb R^d$,
\begin{equation*}
\bigg(\frac{1}{\omega(\mathcal{Q})}\int_{\mathcal{Q}}
\Big[f(x)-\underset{y\in\mathcal{Q}}{\mathrm{ess\,inf}}\,f(y)\Big]^p\omega(x)\,dx\bigg)^{1/p}
\leq C[\omega]_{A^{\rho,\infty}_p}\bigg(1+\frac{r}{\rho(x_0)}\bigg)^{\mathcal{N}}
\end{equation*}
holds true for some $\mathcal{N}>0$.
\item $f\in \mathcal{C}^{\beta,\ast}_{\rho,\infty}(\mathbb R^d)$ with $0<\beta<1$ if and only if there exists a constant $C>0$ such that, for any ball $\mathcal{B}=B(x_0,r)\subset\mathbb R^d$,
\begin{equation*}
\frac{1}{|\mathcal{B}|^{\beta/d}}\bigg(\frac{1}{\omega(\mathcal{B})}\int_{\mathcal{B}}
\Big[f(x)-\underset{y\in\mathcal{B}}{\mathrm{ess\,inf}}\,f(y)\Big]^p\omega(x)\,dx\bigg)^{1/p}
\leq C[\omega]_{A^{\rho,\infty}_p}\bigg(1+\frac{r}{\rho(x_0)}\bigg)^{\mathcal{N'}}
\end{equation*}
holds true for some $\mathcal{N'}>0$.
\end{enumerate}
\end{cor}

\begin{cor}\label{cor4}
Let $1\leq p<q<\infty$ and $\omega\in A^{\rho,\infty}_{p,q}(\mathbb R^d)$. Then the following statements are true.
\begin{enumerate}
  \item $f\in \mathrm{BLO}_{\rho,\infty}(\mathbb R^d)$ if and only if there exists a constant $C>0$ such that, for any cube $\mathcal{Q}=Q(x_0,r)\subset\mathbb R^d$,
\begin{equation*}
\bigg(\frac{1}{\omega^q(\mathcal{Q})}\int_{\mathcal{Q}}
\Big[f(x)-\underset{y\in\mathcal{Q}}{\mathrm{ess\,inf}}\,f(y)\Big]^q\omega(x)^q\,dx\bigg)^{1/q}\leq C[\omega]_{A^{\rho,\infty}_{p,q}}\bigg(1+\frac{r}{\rho(x_0)}\bigg)^{\mathcal{N}}
\end{equation*}
or
\begin{equation*}
\frac{|\mathcal{Q}|^{1/p-1/q}}{[\omega^p(\mathcal{Q})]^{1/p}}\bigg(\int_{\mathcal{Q}}
\Big[f(x)-\underset{y\in\mathcal{Q}}{\mathrm{ess\,inf}}\,f(y)\Big]^q\omega(x)^q\,dx\bigg)^{1/q}\leq C[\omega]_{A^{\rho,\infty}_{p,q}}\bigg(1+\frac{r}{\rho(x_0)}\bigg)^{\mathcal{N}}
\end{equation*}
holds true for some $\mathcal{N}>0$.
  \item $f\in\mathcal{C}^{\beta,\ast}_{\rho,\infty}(\mathbb R^d)$ with $0<\beta<1$ if and only if there exists a constant $C>0$ such that, for any ball $\mathcal{B}=B(x_0,r)\subset\mathbb R^d$,
\begin{equation*}
\frac{1}{|\mathcal{B}|^{\beta/d}}\bigg(\frac{1}{\omega^q(\mathcal{B})}\int_{\mathcal{B}}
\Big[f(x)-\underset{y\in\mathcal{B}}{\mathrm{ess\,inf}}\,f(y)\Big]^q\omega(x)^q\,dx\bigg)^{1/q}\leq C[\omega]_{A^{\rho,\infty}_{p,q}}\bigg(1+\frac{r}{\rho(x_0)}\bigg)^{\mathcal{N'}}
\end{equation*}
or
\begin{equation*}
\frac{|\mathcal{B}|^{1/p-1/q-\beta/d}}{[\omega^p(\mathcal{B})]^{1/p}}
\bigg(\int_{\mathcal{B}}\Big[f(x)-\underset{y\in\mathcal{B}}{\mathrm{ess\,inf}}\,f(y)\Big]^q\omega(x)^q\,dx\bigg)^{1/q}\leq C[\omega]_{A^{\rho,\infty}_{p,q}}\bigg(1+\frac{r}{\rho(x_0)}\bigg)^{\mathcal{N'}}
\end{equation*}
holds true for some $\mathcal{N'}>0$.
\end{enumerate}
\end{cor}

% ------------------------------------------------------------------------
\end{document}